\documentclass[11pt,leqno]{article}
\usepackage[english]{babel}

\usepackage{amsmath}
\usepackage{amssymb}
\usepackage{xypic}
\usepackage[all]{xy}
\UseComputerModernTips
\usepackage{latexsym}
\usepackage{amsfonts}
\usepackage{mathabx}
\usepackage{amsxtra}
\usepackage{mathrsfs}

\usepackage[dvipdfmx]{graphicx}

\usepackage{amsthm,enumerate}

\numberwithin{equation}{section}

\newcommand{\CC}{\mathbb{C}}

\newcommand{\OO}{\mathcal{O}}

\newcommand{\rh}{\mathit{R}\mathcal{H}\mathit{om}}

\newcommand{\id}{\mathrm{id}}

\newcommand{\supp}{\mathrm{supp}}

\newcommand{\imin}[1]{#1^{-1}}

\newcommand{\lind}[1]{\underset{#1}{\underrightarrow{\lim}}}

\newcommand{\xtimes}{\underset{X}{\times}}
\newcommand{\muhom}{\mu hom}
\newcommand{\csapd}{{\chi!!}}
\newcommand{\rcc}{{\mathbb{R}-c}}
\newcommand{\tYX}{{Y \times X}}

\newtheorem{teo}{Theorem}[section]

\newtheorem{df}[teo]{Definition}

\newtheorem{cor}[teo]{Corollary}

\newtheorem{oss}[teo]{Remark}

\newtheorem{prop}[teo]{Proposition}

\newtheorem{lem}[teo]{Lemma}

\newtheorem{es}[teo]{Example}

\usepackage{latexsym}

\title{$\mu hom$ and multi-microlocal operators}
\author{
Naofumi \textsc{Honda}
$\,\,$and$\,\,$ Luca \textsc{Prelli}
}

\date{}

\begin{document}
\maketitle
\begin{abstract}
In this paper, we construct the multi-microlocalization functor
$\mu hom_{\hat{\chi}}$
of homomorphisms, which is a counterpart of the functor $\mu hom$ studied
in M.~Kashiwara and P.~Schapira \cite{KS90}.
Furthermore, using the new functor, 
we also introduce several sheaves of multi-microlocal operators
which act on multi-microlocalized objects such as a multi-microfunction.
\end{abstract}

\tableofcontents

\addcontentsline{toc}{section}{\textbf{Introduction}}

\

Microlocal analysis initiated by M.~Sato, T.~Kawai and M.~Kashiwara (\cite{SKK})
had great impact on several areas of mathematics such as the theory of partial differential equations, that of cotangent geometry, etc. The main tools of microlocal analysis are specialization and microlocalization of a sheaf along a submanifold (see M.~Kashiwara and 
P.~Schapira \cite{KS90}).
Later these tools are extended to sheaves on subanalytic sites by the one of authors (\cite{Pr13}) and to ind-sheaves (M.~Kashiwara, P.~Schapira, F.~Ivorra, I.~Waschkies \cite{KSIW}), and we can now perform microlocal analysis to an object
which is not necessarily a classical sheaf, for example, the family of tempered holomorphic functions and the one of Whitney holomorphic functions.

However, some important objects appearing in analysis require specialization or microlocalization along several submanifolds at once such as strongly asymptotic developability by
H.~Majima (\cite{Ma84}), the simultaneous microlocalization 
by J.~M.~Delort (\cite{De96}) and
the small second microlocalization along two manifolds $M_1, M_2$  
with the totally ordered configuration $M_1 \subset M_2$ 
by P.~Schapira and K.~Takeuchi (\cite{ST94}) and K.~Takeuchi (\cite{Ta96}).
The authors and S.~Yamazaki (\cite{HP13}, \cite{HP24}, \cite{HPY}) recently develop 
the theory of multi-specialization and multi-microlocalization,
by which we can consider specialization and microlocalization 
of a sheaf on a subanalytic site along several submanifolds situated in a good 
configuration. By using these tools the objects mentioned above can 
be managed in a uniform way.
Furthermore applying these tools to the sheaves of holomorphic functions with
several growth orders, we have many new sheaves of strongly asymptotically developable 
functions and those of multi-microfunctions.

One of fundamental questions is to find a class of (multi-)microlocal operators which
can act on these new objects microlocally. The purpose of this paper is to construct the
functor $\mu hom_{\widehat{\chi}}$ 
in the category of multi-microlocalization that is a counterpart 
of the functor $\mu hom$ studied in \cite{KS90}, and then, to establish the theory of
multi-microlocal operators which act on multi-microlocalized objects.
Note that, for a normal crossing case and a totally ordered case with two manifolds already mentioned,
there were pioneering works for multi-microlocal operators by \cite{De96} and H.~Koshimizu and K.~Takeuchi (\cite{KT01}).
To accomplish these purposes in a general framework, 
we need to introduce several new notions, 
some of them are ``a partial order of forest type'', ``the universal family of submanifolds'', 
``a stable subset of ordered type'' 
and ``a compatible family of submersive morphisms'':

We briefly explain these new notions. A partial order $(\Lambda, \preceq)$ 
of forest type represents
an equivalent class of configurations of submanifolds. 
Each partial order $(\Lambda,\preceq)$ 
has a universal family $\widehat{\chi}$ of submanifolds 
which is universal against another family $\chi$
of submanifolds having the same ordered type $(\Lambda, \preceq)$
in the sense that 
the zero section $S^*_{\widehat{\chi}}$ of the multi-normal deformation along
the universal family $\widehat{\chi}$ is big enough so that 
the canonical projection from a submanifold 
in $ S^*_{\widehat{\chi}}$ to the zero section $S^*_\chi$ of
the multi-normal deformation along $\chi$ 
always exists (see Theorem {\ref{thm:muhom_and_mu}} also).
Then we define the functor $\mu hom_{\widehat{\chi}}$ in
the category of multi-microlocalization by using the usual multi-microlocalization
functor $\mu_{\widehat{\chi}}$ 
introduced in \cite{HPY}
along the universal family $\widehat{\chi}$.
Hence $\mu hom_{\widehat{\chi}}$ depends only on the universal family and gives an object on
the zero section $S^*_{\widehat{\chi}}$ of the universal family $\widehat{\chi}$.

Now we can construct multi-microlocal operators by applying $\mu hom_{\widehat{\chi}}$ 
to the sheaf of
holomorphic forms. It is known that 
the multi-microlocal operators have hierarchy so called 
``higher microlocalization''.  
In our framework of multi-microlocal operators,
this hierarchy is determined by a structure associated with the notions 
``a stable subset of ordered type'' and
``a compatible family of submersive morphisms''. In particular,
a submanifold consisting of fibers of each submersive morphism plays
an important role as a base submanifold of a higher microlocalization.
Examples \ref{es:t-s-op} and \ref{es:t-s-ac} typically explain these situations.

\

The paper is organized as follows:
Notions of ``a partial order of forest type`` and ''a universal family of submanifolds``
are explained in Section \ref{sec:universal}. Then, in Section {\ref{section:functor}},
we construct the functor $\mu hom_{\widehat{\chi}}$ and study its properties.
Sheaves of multi-microlocal operators and their actions are studied in
Section {\ref{sec:multil-micro-op}}.

%

\section{A universal family of submanifolds}{\label{sec:universal}}
Through the paper, we use the same notations as those in \cite{KS90}, 
\cite{HPY} and \cite{Pr13}. For the theory of sheaves on subanalytic sites
and that of multi-microlocalization, 
refer the readers \cite{Pr08} and \cite{HPY}, respectively.

Let $X$ be a real analytic manifold, and
let $\chi=\{M_1,\cdots, M_\ell\}$ be a family of closed real analytic submanifolds.
First recall the conditions H1, H2 and H3 for $\chi$ given in p.435 \cite{HP13}:
\begin{enumerate}
\item[H1.] Each $M_j$ is connected and $M_j \ne M_i$ for $i \ne j \in \{1,2,\cdots,\ell\}$.
\item[H2.] For any point $p \in X$, there exist a system
	$(x_1,\cdots,x_n)$ of local coordinates near $p$ and subsets $I_j \subset \{1,2,\cdots,n\}$
($j=1,2,\cdots,\ell$) which satisfy the conditions below:
\begin{enumerate}
	\item $(M_j,\,p) = (\{x_i = 0\,(i \in I_j)\},\, 0)$ holds locally for $j\in \{1,2,\cdots,\ell\}$ with $p \in M_j$,
\item Either $I_{j_1} \subset I_{j_2}$, $I_{j_2} \subset I_{j_1}$ or $I_{j_1} \cap I_{j_2} = \emptyset$ holds for any $j_1, j_2 \in \{1,2,\cdots,\ell\}$.
\end{enumerate}
\item[H3.] For any $j$, we have
$$
M_j\,\, \ne\,\,
\iota_{\chi}(M_j) := \bigcap_{\{M_k \in \chi;\,M_j \subsetneq M_k\}} M_k.
$$
Here if $\{M_k \in \chi;\,M_j \subsetneq M_k\}$ is empty, then we set
$\iota_{\chi}(M_j) = X$. 
\end{enumerate}

Note that, in this paper, we study specialization and microlocalization along a family $\chi$ of submanifolds which satisfies the above conditions H1, H2 and H3 as we did in \cite{HP13} and \cite{HPY}.

\subsection{A partial order of forest type}

Let $\Lambda= \{1,2,\cdots,\ell\}$ and $\preceq$ a partial order on $\Lambda$.
Hereafter, for $i,j \in \Lambda$, 
we write $i \precneq j$ if $i \preceq j$ and $i \ne j$ hold
and $i \nmid j$ if $i$ and $j$ are incomparable.
For $j \in \Lambda$, we set
\begin{equation}
\Lambda_{\preceq j} = \{i \in \Lambda\,|\, i \preceq j\}
\text{ and }
\Lambda_{\succeq j} = \{i \in \Lambda\,|\, i \succeq j\}.
\end{equation}
\begin{df}
We say that $\preceq$ is a partial order of forest type on $\Lambda$ if 
$\preceq$ becomes totally ordered on the subset $\Lambda_{\preceq j}$
for any $j \in \Lambda$.
\end{df}

\begin{df}
Let $\preceq_k$ $(k=1,2)$ be a partial order of forest type on $\Lambda$.
We say that both the partial orders are equivalent if there exists an ordered isomorphism
between $(\Lambda, \preceq_1)$ and $(\Lambda, \preceq_2)$.
\end{df}
We sometimes call an equivalent class $(\Lambda,\preceq)$ of partial orders of forest type 
just ``type''.

Let $E$ be a real vector space with a system $(x^{(0)};\, x^{(1)},\cdots,x^{(\ell)})$  of coordinates of $E$,
where each $x^{(k)}$ denotes a block consisting of coordinate variables of $E$. 
Assume that the coordinate block $x^{(k)}$ is not empty for any $k \in \Lambda$.
Then define
$$
M_j := \{(x^{(0)};\, x^{(1)},\cdots,x^{(\ell)}) \in E\,|\, x^{(i)} = 0
\,(j \preceq i)\} \qquad (j \in \Lambda).
$$
\begin{lem}{\label{lem:partial_order_submanifolds}}
Under the above situation, the family of submanifolds $\{M_j\}_{j\in \Lambda}$ 
in $E$ satisfies the conditions H1, H2 and H3. 
Conversely, let $\{N_j\}_{j \in \Lambda}$ be
a family of closed submanifolds in $E$ satisfying the conditions H1, H2 and H3.
Then the partial order $\preceq$ on $\Lambda$ determined by
$$
i \preceq j \iff N_i \subset N_j
$$
is the one of forest type.
\end{lem}
\begin{es}
Let $\Lambda = \{1,2,\cdots,\ell\}$. Let us consider the partial order $\preceq$ such that
any $i \ne j \in \Lambda$ are incomparable. Then such a $\preceq$ is of forest type and
it is often called ``normal crossing type''.

On the other hand, any total order on $\Lambda$ is also of forest type and
it is called ``totally ordered type''.
\end{es}
\begin{es}
Let $\Lambda= \{1,2,3\}$. The partial order $\preceq$ is determined so that
$3 \precneq 1$, $3 \precneq 2$ and the pair $1$ and $2$ is incomparable.
Then $\preceq$ is of forest type.
\end{es}

Let $X$ be a real analytic manifold and $(\Lambda, \preceq)$ a ``type''.
Set $X^2 = X \times X$ and denote by $\Delta_X$ the diagonal set in $X^2$.
We define a universal family 
\begin{equation}
\widehat{\chi}(X,\,(\Lambda,\preceq)) :=
\{\widehat{X}_1, \widehat{X}_2,\cdots, \widehat{X}_\ell\}
\end{equation}
with the base manifold $X$ associated 
with the type $(\Lambda,\preceq)$ as follows:
Set
$$
X_{k,j} :=
\begin{cases}
\Delta_X \quad & (k \preceq j), \\
X^2 \quad & (\text{otherwise})
\end{cases}
$$
and
$$
\widehat{X}_k := X_{k,1} \times \cdots \times X_{k,\ell} \subset (X^2)^\ell
\qquad (k \in \Lambda).
$$
Note that the universal family of the base manifold $X$ is a family
of closed submanifolds in $(X^2)^\ell$. 
For the universal family $\widehat{\chi} = \widehat{\chi}(X,\,(\Lambda,\preceq))$,
we can consider the multi-normal deformation 
$\widetilde{(X^2)^\ell}$ of $(X^2)^\ell$ along
$\widehat{\chi}$ as in \cite{HP13}. It follows from Proposition 1.2 \cite{HP13} that
the zero section $S_{\widehat{\chi}}$ of $\widetilde{(X^2)^\ell}$ is given by
$$
S_{\widehat{\chi}} =
T_{\Delta_X}X^2 \times T_{\Delta_X}X^2 \times \cdots \times T_{\Delta_X}X^2
= (TX)^\ell
$$
and hence, its dual vector bundle is
$$
S^*_{\widehat{\chi}} = 
T^*_{\Delta_X}X^2 \times T^*_{\Delta_X}X^2 \times \cdots \times T^*_{\Delta_X}X^2
= (T^*X)^\ell.
$$
We also define the vector bundle over $X$ by
$$
S^*_{\widehat{\chi},X} = 
X \times_{X^\ell} S^*_{\widehat{\chi}} =
X \times_{X^\ell} (T^*X)^\ell.
$$
Here $X$ is identified with the diagonal of $X^\ell$.
\begin{es}
Let us consider examples on $\Lambda = \{1,2,3\}$. 
\begin{itemize}
\item Let us consider the partial order $\preceq$ such that
any $i \ne j \in \Lambda$ are incomparable (normal crossing type). In this case we have
\begin{eqnarray*}
\widehat{X}_1 & = & \Delta_X \times X^2 \times X^2, \\
\widehat{X}_2 & = & X^2 \times \Delta_X \times X^2, \\
\widehat{X}_3 & = & X^2 \times X^2 \times \Delta_X.
\end{eqnarray*}
\item Let us consider a total order $\preceq$ on $\Lambda$ such that $3 \precneq 2 \precneq 1$ (totally ordered type). In this case we have
\begin{eqnarray*}
\widehat{X}_1 & = & \Delta_X \times X^2 \times X^2, \\
\widehat{X}_2 & = & \Delta_X \times \Delta_X \times X^2, \\
\widehat{X}_3 & = & \Delta_X \times \Delta_X \times \Delta_X.
\end{eqnarray*}
\item Let us consider a partial order $\preceq$ on $\Lambda$ such that
$3 \precneq 1$, $3 \precneq 2$ and the pair $1$ and $2$ is incomparable. In this case we have
\begin{eqnarray*}
\widehat{X}_1 & = & \Delta_X \times X^2 \times X^2, \\
\widehat{X}_2 & = &  X^2 \times \Delta_X \times X^2, \\
\widehat{X}_3 & = & \Delta_X \times \Delta_X \times \Delta_X.
\end{eqnarray*}
\end{itemize}
For these universal families, we have
$$
S^*_{\widehat{\chi}} = T^*X \times T^*X \times T^*X,\qquad
S^*_{\widehat{\chi},X} = T^*X \times_X T^*X \times_X T^*X.
$$
\end{es}

By taking Lemma \ref{lem:partial_order_submanifolds} into account,
a family  of closed submanifolds in $X$
satisfying the conditions H1, H2 and H3 is also called of forest type.
Let $\chi=\{M_1,\cdots, M_\ell\}$ be a family of closed submanifolds of forest type in $X$.
It follows from the conditions H1, H2 and H3 that, for any $i \ne j \in \{1,2,\cdots,\ell\}$, only one of the following configurations of submanifolds occurs:
\begin{enumerate}
\item $M_i \subsetneq M_j$.
\item $M_j \subsetneq M_i$.
\item $M_i$ and $M_j$ are transversally intersecting at any point in $M_i \cap M_j$.
\end{enumerate}
By the same reasoning as in Lemma \ref{lem:partial_order_submanifolds}, the partial order
$\preceq$ on $\Lambda$
$$
i \preceq j \iff M_i \subset M_j 
$$
becomes of forest type, which is called the type induced from $\chi$.
Then we obtain the universal family of closed submanifolds in $(X^2)^\ell$ 
associated with the induced type of $\chi$.  
Such a universal family is often denoted by $\widehat{\chi}$.

\subsection{A stable subset}
Let $(\Lambda,\preceq)$ be a ``type'' and $\Theta$ a subset in $\Lambda$.  
Note that $(\Theta,\,\preceq|_{\Theta})$ also defines a type.

\begin{df}{\label{df:stable_subset}}
A subset $\Theta$ is called stable if the following conditions are 
satisfied:
\begin{enumerate}
\item For any $i \in \Theta$, 
there exists $j \in \Lambda$ with $j \precneq i$.
\item For any $i \in \Theta$ and $j \in \Lambda$, 
	if $i \preceq j$, then $j \in \Theta$.
\end{enumerate}
\end{df}

%
%

We define a subset $\mathrm{m}^*(\Lambda,\Theta)$ of $\Lambda \setminus \Theta$ 
as follows: $j \in \mathrm{m}^*(\Lambda,\Theta)$ if and only if 
$j \in \Lambda \setminus \Theta$
and satisfies the one of the conditions below.
\begin{enumerate}
\item There exists no $j' \in \Lambda \setminus \Theta$ with $j \precneq j'$.
\item There exists an $i \in \Theta$ such that $j \precneq i$ and there is
no $j' \in \Lambda \setminus \Theta$ with $j \precneq j' \precneq i$.
\end{enumerate}
\begin{lem}{\label{lem:m-star-1}}
Let $j \in \Lambda \setminus \Theta$.  Then there exists $j \preceq k$
such that $k \in \mathrm{m}^*(\Lambda,\Theta)$.
\end{lem}
\begin{proof}
Take the maximal element $k$ in $\{i \in \Lambda \setminus \Theta\,|\, j \preceq i\}$.
Then we again take the maximal element $k'$ in $\{i \in \Lambda\,|\, k \preceq i\}$.
If $k = k'$, we get $k \in \mathrm{m}^*(\Lambda,\Theta)$ by
the condition 1.~in the definition of $\mathrm{m}^*(\Lambda,\Theta)$. 
Otherwise, we have $k' \in \Theta$, and hence we conclude
$k \in \mathrm{m}^*(\Lambda,\Theta)$ 
by the condition 2.~in the definition of $\mathrm{m}^*(\Lambda,\Theta)$. 
\end{proof}

\begin{lem}{\label{lem:m-star-2}}
Let $j \in \Theta$.  Then there exists $k \preceq j$
such that $k \in \mathrm{m}^*(\Lambda,\Theta)$.
\end{lem}
\begin{proof}
Consider the subset 
$$
R = \Lambda_{\preceq j} \cap (\Lambda \setminus \Theta)
=\{i \in \Lambda \setminus \Theta\,|, i \preceq j\}.
$$
Then $R$ is non-empty because $\Theta$ is stable and it is totally ordered. 
It follows from 
the condition 2.~in the definition of $\mathrm{m}^*(\Lambda,\Theta)$ that
the maximum of $R$ belongs to $\mathrm{m}^*(\Lambda,\Theta)$.
\end{proof}

\begin{es}{\label{ex:t-s-type}}
Let $\Lambda = \{1,2,\cdots,\ell\}$ and let $(\Lambda,\preceq)$ 
be of totally ordered type, that is,
$$
\ell \precneq \ell - 1 \precneq \cdots \precneq 2 \precneq 1.
$$
For $0 \le d < \ell$, set
$$
\Theta=\{1,\cdots, d\}.
$$
Then $\Theta$ is a stable subset and $\mathrm{m}^*(\Lambda, \Theta) = \{d+1\}$.
\end{es}

\begin{es}{\label{ex:mixed-type}}
Let $\Lambda = \{1,2,3\}$. The partial order $\preceq$ on $\Lambda$ is defined by
$$
3 \precneq 1, \quad 3 \precneq 2, \quad 1 \nmid 2.
$$
Note that $(\Lambda, \preceq)$ is of forest type.

Let $\Theta = \{2\}$.  Then $\Theta$ is a stable subset 
and $\mathrm{m}^*(\Lambda, \Theta) = \{1, 3\}$.
\end{es}

\begin{es}
Let $\Lambda = \{1,2,\cdots,\ell\}$ and let $(\Lambda,\preceq)$ 
be of normal crossing type, that is,
$$
i \nmid j \qquad (i \ne j \in \Lambda).
$$
Let $\Theta = \emptyset$.  Then $\Theta$ is a stable subset 
and $\mathrm{m}^*(\Lambda, \Theta) = \Lambda$.
\end{es}

\section{The functor $\mu hom_\chi$}{\label{section:functor}}
Let $f:X \to Y$ be a morphism of real analytic manifolds,
and let $\chi^M=\{M_1,\dots,M_\ell\}$, $\chi^N=\{N_1,\dots,N_\ell\}$ be 
two families of closed submanifolds of forest type in $X$ and $Y$, respectively.
We assume the conditions below:
\begin{enumerate}
\item $f(M_j) \subseteq N_j$, $j=1,\dots,\ell$,
\item $M_j \subseteq M_{j'}$ iff $N_j \subseteq N_{j'}$.
\end{enumerate}
Note that the condition 2.~is equivalently saying that the types induced from 
$\chi^M$ and $\chi^N$ coincide.
Set $M=M_1\cap \cdots \cap M_\ell$ and $N=N_1 \cap \cdots \cap N_\ell$.
In this situation $f$ induces a morphism of vector bundles 
$$
f':S_{\chi^M}=\times_M T_{M_j}\iota(M_j) \to M \times_Y (\times_N T_{N_j}\iota(N_j)) = 
M \times_Y S_{\chi^N}
$$ 
and the map
$$
f_\tau:  M \times_Y S_{\chi^N} \to S_{\chi^N},
$$
where $\iota$ was defined in the condition H3 in Section \ref{sec:universal}
and $\times_M T_{M_j}\iota(M_j)$ denotes $T_{M_1}\iota(M_1) \times_X \cdots \times_X
T_{M_\ell}\iota(M_\ell)$.
With the same notations of \cite{KS90} we write $T_{\chi}f=f_\tau \circ f':S_{\chi^M} \to S_{\chi^N}$. The two propositions below have already been proved in \cite{HP13} and
\cite{HP24} for more general cases. 
For a non-empty subset $J \subset \{1,2,\cdots,\ell\}$, the $\chi^M_J$ denotes
the subfamily $\{M_j\}_{j \in J}$ of $\chi^M$, which is still of forest type.

\begin{prop}[Propostion 3.18 \cite{HP24}]{\label{prop:multimicro-proper}} 
Let $F \in D^b(k_{X_{sa}})$.
\begin{itemize}
\item[(i)] There exists a commutative diagram of canonical
morphisms
$$
\xymatrix{R(T_\chi f)_{\csapd}\nu^{sa}_{\chi^M} F  \ar[d] \ar[r] & \nu^{sa}_{\chi^N}Rf_{!!}F \ar[d] \\
R(T_\chi f)_*\nu^{sa}_{\chi^M} F & \nu^{sa}_{\chi^N}Rf_*F. \ar[l]}
$$
\item[(ii)] Moreover if $f:\supp F \to Y$ and $T_{\chi_J}f:C_{\chi^M_J}(\supp F) 
\to S_{\chi^N_J}$
for each non-empty $J \subset \{1,\dots,\ell\}$ are proper, and if $\supp F \cap \imin f(N_j) \subseteq M_j$, $j \in \{1,\ldots,\ell\}$, then
the above morphisms are isomorphisms.
\end{itemize}
\end{prop}
Here for a multi-conic map $h: S_{\chi^M} \to S_{\chi^N}$,
the functor $h_{\csapd}$ denotes a proper direct image functor
in the categories of multi-conic sheaves on subanalytic sites. For the details, see Subsection 2.5 in \cite{Pr13} and Appendix in \cite{HP13}. That is, for a multi-conic sheaf $G$
on $(S_{\chi^M})_{sa}$ and a multi-conic open subanalytic subset $U \subset S_{\chi^N}$
with $\tau_N(U)$ being relatively compact 
($\tau_N: S_{\chi^N} \to N$ 
and $\tau_M: S_{\chi^M} \to M$ 
are the canonical projections),
\begin{equation}
\Gamma(U;\,h_{\csapd} G) = \lind{Z,K}\,\Gamma_{\tau_M^{-1}(K) \cap Z}(h^{-1}(U);\,G),
\end{equation}
where $Z$ runs through multi-conic closed subanalytic subsets in $h^{-1}(U)$ such that
$h|_{Z}: Z \to U$ is proper over $U$ and $K$ runs through compact subsets in 
$M$.

\begin{prop}[Proposition 3.19 \cite{HP24}] Let $F \in D^b(k_{Y_{sa}})$.
\begin{itemize}
\item[(i)] There exists a commutative diagram of canonical
morphisms
$$
\xymatrix{\omega_{S_{\chi^M}/S_{\chi^N}}\otimes\imin {(T_\chi f)}\nu^{sa}_{\chi^N} F  \ar[d] \ar[r] & \nu^{sa}_{\chi^M}(\omega_{X/Y}\otimes \imin f F) \ar[d] \\
T_{\chi} f^!\nu^{sa}_{\chi^N} F & \nu^{sa}_{\chi^M}f^!F. \ar[l]}
$$
\item[(ii)]
The above morphisms are isomorphisms on the open sets where $T_{\chi} f$ is smooth.
\end{itemize}
\end{prop}

The proof of the following result goes as Proposition 4.2.6 of \cite{KS90}. 
Remark that the family $\{M_1 \times N_1,\dots,M_\ell \times N_\ell\}$ has the same type as the one of $\chi^M$ and $\chi^N$.

\begin{prop} Let $F \in D^b(\CC_{X_{sa}})$ and $G \in D^b(\CC_{Y_{sa}})$. There is a natural morphism
$$
\nu^{sa}_{\chi^M}F \boxtimes \nu^{sa}_{\chi^N}G \to \nu^{sa}_{\chi^{M \times N}}(F \boxtimes G).
$$
where $\chi^{M \times N}=\{M_1 \times N_1,\dots,M_\ell \times N_\ell\}$.
\end{prop}

Note that Corollary 4.2.7 \cite{KS90} can be obtained as well, but we do not need it in this paper.

Now, let us consider the multi-microlocalization. As for the classical case, 
we call ${}^tf'$ and $f_\pi$ the maps obtained from $f'$ and $f_\tau$, respectively. The following propositions are obtained by applying multi-Fourier-Sato transforms to the previous results for multi-specialization combined with Proposition 3.2.9, 3.2.10, 3.2.11 of \cite{Pr13}.

\begin{prop}{\label{prop:direct_image}} Let $F \in D^b(\CC_{X_{sa}})$.
\begin{itemize}
\item[(i)] There exists a commutative diagram of canonical
morphisms
$$
\xymatrix{Rf_{\pi \csapd}{}^tf'{}^{-1}\mu^{sa}_{\chi^M} F  \ar[d] \ar[r] & \mu^{sa}_{\chi^N}Rf_{!!}F \ar[d] \\
Rf_{\pi *}({}^tf'{}^{!}\mu^{sa}_{\chi^M} F \otimes \omega_{X/Y} \otimes \omega_{M/N}^{\otimes -1}) & \mu^{sa}_{\chi^N}Rf_*F. \ar[l]}
$$
\item[(ii)] Moreover if the same conditions as in Proposition {\ref{prop:multimicro-proper}} (ii) are satisfied, then the above morphisms are isomorphisms.
\end{itemize}
\end{prop}
Note that the conditions in (ii) of the above proposition hold if
the following conditions are satisfied:
\begin{enumerate}
\item $f$ is proper on $\mathrm{supp}(F)$.
\item $M_k = f^{-1}(N_k)$ holds for each $k \in \{1,2,\cdots,\ell\}$.
\item For each non-empty $J \subset \{1,\cdots,\ell\}$ and $j \in J$,
the submanifold $N_j$ is clean for the restriction map 
$f|_{\iota_{\chi^M_J}(M_j)}: \iota_{\chi^M_J}(M_j) \to \iota_{\chi^N_J}(N_j)$.
\end{enumerate}

\

\begin{prop}{\label{prop:inverse_image}} Let $F \in D^b(\CC_{Y_{sa}})$.
\begin{itemize}
\item[(i)] There exists a commutative diagram of canonical
morphisms
$$
\xymatrix{R{}^tf'_\csapd (\omega_{M/N}\otimes\imin {f_\pi}\mu^{sa}_{\chi^N} F)  \ar[d] \ar[r] & \mu^{sa}_{\chi^M}(\omega_{X/Y}\otimes \imin f F) \ar[d] \\
R{}^tf'_* f_\pi^!\mu^{sa}_{\chi^N} F & \mu^{sa}_{\chi^M}f^!F. \ar[l]}
$$
\item[(ii)]
The above morphisms are isomorphisms if $T_{\chi} f$ is smooth.
\end{itemize}
\end{prop}
Note that the condition of (ii) in the above proposition is satisfied if $f:X \to Y$ and
each $f|_{M_j}: M_j \to N_j$ ($j =1,2,\dots,\ell$) are smooth.

\

\begin{prop} Let $F \in D^b(\CC_{X_{sa}})$ and $G \in D^b(\CC_{Y_{sa}})$. There is a natural morphism
$$
\mu^{sa}_{\chi^M}F \boxtimes \mu^{sa}_{\chi^N}G \to \mu^{sa}_{\chi^{M \times N}}(F \boxtimes G),
$$
where $\chi^{M \times N}=\{M_1 \times N_1,\dots,M_\ell \times N_\ell\}$.
\end{prop}

\

Let $\Delta_f$ be the graph of $f$. Consider the following diagram
\begin{equation}
\xymatrix{
X \times X \ar[r]^{f_1} & Y \times X \ar[r]^{f_2} & Y \times Y \\
\Delta_X \ar[u] \ar[r] & \Delta_f \ar[u] \ar[r] & \Delta_Y, \ar[u]
}
\end{equation}
where $f_1 = f \times \operatorname{id}_X$ and $f_2 = \operatorname{id}_Y \times f$.

Let $\chi_{X \times X} = \{M_1,\cdots,M_\ell\}$ and $\chi_{Y \times Y} = \{L_1,\cdots,L_\ell\}$ be families of closed submanifolds of forest type in $X \times X$ and $Y \times Y$ 
which satisfy the following conditions:
\begin{enumerate}
\item[A.] The induced types of $\chi_{X \times X}$ and $\chi_{Y \times Y}$ coincide
and $(f \times f)(M_k) \subset L_k$ for any $k \in \{1,2,\cdots,\ell\}$.
\item[B.] $\bigcap_j M_j = \Delta_X$ and $\bigcap_j L_j = \Delta_Y$.
\end{enumerate}
Set
$$
\chi_{\tYX} = \{N_1,\cdots, N_\ell\} := \{f_2^{-1}(L_1),\cdots, f_2^{-1}(L_\ell)\}.
$$
Then we have the following facts:
\begin{lem}
We have $\bigcap_{j} N_j = \Delta_f$ and
$\chi_{\tYX}$ is a family of closed submanifolds of forest type
in an open neighborhood of $\Delta_f$.
\end{lem}
\begin{proof}
The fact $\bigcap_{j} N_j = \Delta_f$ easily follows from the definition. Let $q \in \Delta_f$ and set $p = f_2(q)$. Note that $p \in \Delta_Y$ holds.
Then there exist real analytic functions 
$\varphi_1(y_1,y_2), \cdots, \varphi_m(y_1,y_2)$ ($m = \mathrm{dim} Y$) 
defined on an open neighborhood of $p$ in $Y \times Y$ satisfying
$$
d\varphi_1(p) \wedge d\varphi_2(p) \wedge \cdots \wedge d\varphi_m(p) \ne 0
$$
and a family of subsets $I_j \subset \{1,2,\cdots,m\}$ ($j=1,\cdots,\ell$) which satisfies the condition (b) of H2 such that
$$
L_k = \{(y_1,y_2) \in Y \times Y;\, \varphi_i(y_1,y_2) = 0\,\,(i \in I_k)\}\qquad (k=1,2,\cdots,\ell)
$$ 
hold near $p$. 

By noticing $\Delta_Y = \cap_j L_j \subset L_k$,  we get
$$
d\varphi_k(p) \in (T^*_{\Delta_Y}(Y \times Y))_p \qquad (k=1,2,\cdots,\ell).
$$
Furthermore, $f_2$ induces the isomorphism
$$
f_2': 
(T^*_{\Delta_Y}(Y \times Y))_p \xrightarrow{\,\,\,\sim\,\,\,}
(T^*_{\Delta_f}(Y \times X))_q.
$$
Hence we have obtained
$$
d(\varphi_1 \circ f_2)(q) \wedge d(\varphi_2\circ f_2)(q) \wedge \cdots \wedge 
d(\varphi_m \circ f_2)(q) \ne 0
$$
from which the later claim of the lemma follows because of
$$
N_k =f^{-1}_2(L_k) = \{(y,x) \in Y \times X;\, (\varphi_i \circ f_2)(y,x) = 0\,\,(i \in I_k)\}\qquad (k=1,2,\cdots,\ell)
$$
near $q$. This completes the proof.
\end{proof}
Note that $f_1$ and $f_2$ satisfy the conditions 1.~and 2.~
stated at the beginning of this section with respect to these families.
Note also that $S^*_{\chi_{X\times X}}$, $S^*_{\chi_{\tYX}}$ and
$S^*_{\chi_{Y \times Y}}$ are vector bundles over
$\Delta_X = X$, $\Delta_f = X$ and $\Delta_Y = Y$, respectively, and
the morphisms ${f_{1\pi}}: \Delta_X \underset{X \times X}{\times} S^*_{\chi_{\tYX}} \to 
S^*_{\chi_{\tYX}}$ 
and $f_2':  S_{\chi_{\tYX}} \to 
\Delta_f \underset{Y \times Y}{\times}	S_{\chi_{Y \times Y}}$
are isomorphic. In particular, we have
$$
S^*_{\chi_{\tYX}} \simeq 
\Delta_f \underset{Y}\times S^*_{\chi_{Y \times Y}}
\simeq
X \underset{Y}\times S^*_{\chi_{Y \times Y}}.
$$
Hence, in what follows, we set
$$
f_\pi = {f_{2 \pi}} \circ ({}^tf_2')^{-1}\qquad \text{and} \qquad {}^t 
f' = {}^t {f_1'} \circ (f_{1\pi})^{-1}. 
$$
This situation is visualized by the diagram
\begin{equation}
S^*_{\chi_{X \times X}} \xleftarrow{\,\,{}^t f'\,\,} 
S^*_{\chi_{\tYX}}
\simeq X \underset{Y}\times S^*_{\chi_{Y \times Y}}
\xrightarrow{\,\,f_\pi\,\,} S^*_{\chi_{Y \times Y}}.
\end{equation}

\

	Let $p_j$ (resp. $\tilde{p}_j$) be the $j$-th projection on $X \times X$ (resp. $Y \times X$).

\begin{df} Let $F \in D^b(\CC_{X_{sa}})$ and let $G \in D^b(\CC_{Y_{sa}})$. We define
\begin{eqnarray*}
\mu hom^{sa}_{\chi_{\tYX}}(F \to G):=\mu^{sa}_{\chi_{\tYX}}\rh(\tilde{p}_2{}^{-1}F,\tilde{p}_1^!G), \\
\mu hom^{sa}_{\chi_{\tYX}}(G \gets F):= (\mu^{sa}_{\chi_{\tYX}}\rh(\tilde{p}_1{}^{-1}G,\tilde{p}_2^!F))^a.
\end{eqnarray*}
In particular, if $X=Y$, $f=\id$ and $\chi_{Y \times Y}=\chi_{\tYX} = \chi_{X \times X}$, we set
$$
\mu hom^{sa}_{\chi_{X \times X}}(F,G) = \mu hom^{sa}_{\chi_{X \times X}}(F \to G) = \mu^{sa}_{\chi_{X \times Y}}\rh({p}_2{}^{-1}F,{p}_1^!G).
$$
\end{df}

Thanks to the above results for multi-specialization and multi-microlocalization the following results are obtained functorially exactly as in $\S$ 4.4 of \cite{KS90}.
Let $F \in D^b(\CC_{X_{sa}})$ and let $G \in D^b(\CC_{Y_{sa}})$. 

\begin{prop} One has commutative diagrams of morphisms
$$
\xymatrix{
R{^t}f'_\csapd\mu hom^{sa}_{\chi_{\tYX}}(F \to G)  \ar[r] \ar[d] & \mu hom^{sa}_{\chi_{X \times X}}(F,\imin f G \otimes \omega_{X/Y}) \ar[d] \\
R{}^tf'_*\mu hom^{sa}_{\chi_{\tYX}}(F \to G) & \mu hom^{sa}_{\chi_{X \times X}}(F,f^!G), \ar[l]
}
$$
$$
\xymatrix{
R{^t}f'_\csapd\mu hom^{sa}_{\chi_{\tYX}}(G \gets F)  \ar[r] \ar[d] & \mu hom^{sa}_{\chi_{X \times X}}(f^!G,F \otimes \omega_{X/Y}) \ar[d] \\
R{}^tf'_*\mu hom^{sa}_{\chi_{\tYX}}(G \gets F) & \mu hom^{sa}_{\chi_{X \times X}}(\imin f G,F), \ar[l]
}
$$
$$
\xymatrix{
Rf_{\pi\csapd}\mu hom^{sa}_{\chi_{\tYX}}(F \to G) \ar[r] \ar[d] & \mu hom^{sa}_{\chi_{Y \times Y}}(Rf_*F,G) \ar[d] \\
Rf_{\pi*}\mu hom^{sa}_{\chi_{\tYX}}(F \to G) & \mu hom^{sa}_{\chi_{Y \times Y}}(Rf_!F,G), \ar[l]
}
$$
$$
\xymatrix{
Rf_{\pi\csapd}\mu hom^{sa}_{\chi_{\tYX}}(G \gets F) \ar[r] \ar[d] & \mu hom^{sa}_{\chi_{Y \times Y}}(G,Rf_!F) \ar[d] \\
Rf_{\pi*}\mu hom^{sa}_{\chi_{\tYX}}(G \gets F) & \mu hom^{sa}_{\chi_{Y \times Y}}(G,Rf_*F). \ar[l]
}
$$
\end{prop}

\begin{prop} One has canonical morphisms
\begin{eqnarray*}
	R{}^tf'_\csapd f_\pi^{-1}\mu hom^{sa}_{\chi_{Y \times Y}}(Rf_!F,G) & \to & \mu hom^{sa}_{\chi_{X \times X}}(F,\imin f G \otimes \omega_{X/Y}), \\
	R{}^tf'_\csapd f_\pi^{-1}\mu hom^{sa}_{\chi_{Y \times Y}}(G,Rf_*F) & \to & \mu hom^{sa}_{\chi_{X \times X}}(f^!G,F \otimes \omega_{X/Y}), \\
	Rf_{\pi \csapd}{}^tf'{}^{-1}\mu hom^{sa}_{\chi_{X \times X}}(F,f^!G) & \to & \mu hom^{sa}_{\chi_{Y \times Y}}(Rf_*F,G), \\
	Rf_{\pi \csapd}{}^tf'{}^{-1}\mu hom^{sa}_{\chi_{X \times X}}(\imin f G,F) & \to & \mu hom^{sa}_{\chi_{Y \times Y}}(G,Rf_!F).
\end{eqnarray*}
\end{prop}

\begin{prop} Let $F_1,F_2 \in D^b(\CC_{X_{sa}})$ and $G_1,G_2 \in D^b(\CC_{Y_{sa}})$. One has commutative diagrams of morphisms
$$
\xymatrix{
R{}^tf'_\csapd f_\pi^{-1}\mu hom^{sa}_{\chi_{Y \times Y}}(G_1,G_2)  \ar[r] \ar[d] & \mu hom^{sa}_{\chi_{X \times X}}(f^!G_1,\imin f G_2 \otimes \omega_{X/Y}) \ar[d] \\
R{}^tf'_*f_\pi^!(\mu hom^{sa}_{\chi_{Y \times Y}}(G_1,G_2) \otimes \omega_{X/Y}^{\otimes-1}) & \mu hom^{sa}_{\chi_{X \times X}}(\imin fG_1 \otimes \omega_{X/Y},f^! G_2), \ar[l]
}
$$
$$
\xymatrix{
Rf_{\pi \csapd}{}^tf'{}^{-1}\mu hom^{sa}_{\chi_{X \times X}}(F_1,F_2) \ar[r] \ar[d] & \mu hom^{sa}_{\chi_{Y \times Y}}(Rf_*F_1,Rf_!F_2) \ar[d] \\
Rf_{\pi*}({}^tf'{}^!\mu hom^{sa}_{\chi_{X \times X}}(F_1,F_2) \otimes \omega_{X/Y}) & \mu hom^{sa}_{\chi_{Y \times Y}}(Rf_!F_1,Rf_*F_2). \ar[l]
}
$$
\end{prop}

\

Let $q_X:X \to S$, $q_Y:Y \to S$ be two morphisms, and denote by $p_1,p_2$ the first and second projection on $Y \times X$ and $Y \times_S X$. 
Assume that $X \times_S Y$ is a closed submanifold in $X \times Y$, and
$j:X \times_S Y \hookrightarrow X \times Y$ denotes a closed embedding.
We also define the closed embeddings
$$
(X \times_S Y)^2 \xrightarrow{\,\,j_1\,\,} (X \times Y) \times (X \times_S Y)
\xrightarrow{\,\,j_2\,\,} (X \times Y)^2 = (X \times X) \times (Y \times Y),
$$
and set $\tilde{j} = j_2 \circ j_1$.

Let $\chi_{X \times X} = \{M_1,\cdots,M_\ell\}$ and 
$\chi_{Y\times Y} = \{N_1,\cdots, N_\ell\}$ 
be families of closed submanifolds of forest type
in $X \times X$ and $Y \times Y$, respectively. 
We also assume that types induced from $\chi_{X \times X}$ and $\chi_{Y \times Y}$ 
coincide and that $\bigcap_j M_j = \Delta_X$ and $\bigcap_j N_j = \Delta_Y$ hold.
Define the family in $(X \times_S Y)^2$ 
by
$$
\chi_{(X \times_S Y)^2}=
\{\tilde{j}^{-1}(M_1 \times N_1), 
\cdots, \tilde{j}^{-1}(M_\ell \times N_\ell)\}.
$$
For which we assume that each $\tilde{j}^{-1}(M_k \times N_k)$ 
becomes a closed submanifold and that
$\chi_{(X\times_S Y)^2}$ is of forest type.

\begin{prop} 
Let $F_1,F_2 \in D^b(\CC_{X_{sa}})$ and $G_1,G_2 \in D^b(\CC_{Y_{sa}})$. 
Under the above situation, one has a morphism
$$
R({}^tj')_{\csapd} \left(\mu hom^{sa}_{\chi_{X \times X}}(F_1,F_2) \boxtimes_S \mu hom^{sa}_{\chi_{Y \times Y}}(G_1,G_2)\right) 
\to \mu hom^{sa}_{\chi_{(X \times_S Y)^2}}(F_1 \boxtimes_S G_1, F_2 \boxtimes_S G_2).
$$
\end{prop}

\

Let $Z$ be another manifold and
let $\chi_{(X\times Y)^2} = \{M_1,\cdots, M_\ell\}$
(resp. $\chi_{(Y \times Z)^2} = \{N_1,\cdots, N_\ell\}$,
$\chi_{(X \times Z)^2} = \{L_1,\cdots, L_\ell\}$)
be a family of closed submanifolds in $(X\times Y)^2$ (resp.
$(Y \times Z)^2$, $(X \times Z)^2$) of forest type.
Assume that these families have the same induced type and
that $\bigcap_j M_j = \Delta_{X \times Y}$,
$\bigcap_j N_j = \Delta_{Y \times Z}$ and
$\bigcap_j L_j = \Delta_{X \times Z}$ hold.

Let us define the family $\chi_{(X\times Y \times Z)^2}$ of submanifolds
in $(X \times Y \times Z)^2$ as follows:
Let us consider the canonical projections $p_{12,2}: (X \times Y) \to Y$ and
$p_{23,2}:(Y \times Z) \to Y$. Then, by noticing 
$$
j: X \times Y \times Z
= (X \times Y) \times_Y (Y \times Z)  \hookrightarrow (X \times Y) \times (Y \times Z),
$$
we define the closed embeddings
$$
(X \times Y \times Z)^2 
\xrightarrow{\,\,j_1\,\,}
((X \times Y) \times (Y \times Z)) \times (X \times Y \times Z)
\xrightarrow{\,\,j_2\,\,}
((X \times Y) \times (Y \times Z))^2 =
(X \times Y)^2 \times (Y \times Z)^2
$$
and set $\tilde{j} = j_2 \circ j_1$.
We define the family in $(X \times Y \times X)^2$
$$
\chi_{(X\times Y \times Z)^2}
=\{\tilde{j}^{-1}(M_1 \times N_1), \cdots, \tilde{j}^{-1}(M_\ell \times N_\ell)\}.
$$
Now assume that all the $\tilde{j}^{-1}(M_k \times N_k)$ are closed submanifolds 
and that $\chi_{(X\times Y \times Z)^2}$ is of forest type.
Furthermore, we assume that
$$
\tilde{p}_{13}(\tilde{j}^{-1}(M_k \times N_k)) \subset L_k\qquad (k=1,2,\cdots,\ell),
$$
where $\tilde{p}_{13}: (X \times Y \times Z)^2 \to (X \times Z)^2$ is the canonical projection.

Note that the projection $p_{13}:X \times Y \times Z \to X \times Z$ induces the map
$$
S^*_{\chi_{(X \times Z)^2}} \xleftarrow{\,\,p_{13\pi}\,\,}
(X\times Y \times Z) \underset{X\times Z}{\times} S^*_{\chi_{(X \times Z)^2}}
\xrightarrow{\,\,{}^t p'_{13}\,\,}
S^*_{\chi_{(X\times Y \times Z)^2}}.
$$
\begin{oss}
There is no canonical projection
$$
q_{13}:S^*_{\chi_{(X\times Y \times Z)^2}} \to 
S^*_{\chi_{(X \times Z)^2}}
$$
under a current general situation. Hence we cannot simplify the left-side of the formula below anymore.
\end{oss}
\begin{prop} Let $F_1,G_1 \in D^b(\CC_{(X \times Y)_{sa}})$ and 
	$F_2,G_2 \in D^b(\CC_{(Y \times Z)_{sa}})$. Under the above situation, there is a canonical morphism
\begin{eqnarray*}
	& R(p_{13\pi})_{\csapd} ({}^tp'_{13})^{-1}R({}^tj')_{\csapd}\left(\mu hom^{sa}_{\chi_{(X \times Y)^2}}(F_1,G_1) 
\boxtimes_{Y} \mu hom^{sa}_{\chi_{(Y \times Z)^2}}(F_2,G_2)\right) & \\
& \to \mu hom^{sa}_{\chi_{(X \times Z)^2}}(F_1 \circ F_2,G_1 \circ G_2). &
\end{eqnarray*}
\end{prop}

%

\

Now we apply these results to the functor $\mu hom^{sa}_{\widehat{\chi}}(F_1,\dots,F_\ell;G)$.
Let $(\Lambda, \preceq)$ be of forest type, and let
$\widehat{\chi}$ denote the universal family $\widehat{\chi}(X, (\Lambda,\preceq))$.
Recall that we have
$$
S^*_{\widehat{\chi}} = T^*X \times \cdots \times T^*X.
$$
We also define the vector bundle $S^*_{\widehat{\chi},X}$ over $X$ by
$$
S^*_{\widehat{\chi},X} = X \underset{X^\ell}{\times} S^*_{\widehat{\chi}} = T^*X \xtimes \cdots \xtimes T^*X,
$$
and denote by $i_{\widehat{\chi}}$ the closed embedding 
$
S^*_{\widehat{\chi},X} \hookrightarrow S^*_{\widehat{\chi}}
$. Note that
we sometimes write $S^*_{\widehat{\chi},\Delta}$ instead of $S^*_{\widehat{\chi},X}$
in this paper.

Let $p_2: (X^2)^\ell \to X^\ell$ (resp.  $p_1: (X^2)^\ell \to X^\ell$)
denote the canonical projection associated with
the one $X^2 = X \times X \ni (x,y) \mapsto y \in X$
(resp.  $X^2 = X \times X \ni (x,y) \mapsto x \in X$), and let
$i_\Delta: X \to X^\ell$ designate
the diagonal embedding, i.e., $i_\Delta(x) = (x, \cdots, x)$.
\begin{df}
Let $F_1,\dots,F_\ell$ and $G$ be in $D^b(\CC_{X_{sa}})$. 
We define the functor
\begin{equation}
\mu hom^{sa}_{\widehat{\chi}}(F_1,\dots,F_\ell;G) =
i_{\widehat{\chi}}^{-1}\mu^{sa}_{\widehat{\chi}}(\rh(p_2^{-1}(F_1 \boxtimes \cdots \boxtimes F_\ell),\,p_1^!Ri_{\Delta*}G)).
\end{equation}
\end{df}
\begin{oss}
For $F_1,\dots,F_\ell$ and $G$ be in $D^b(\CC_{X})$, we also define the functor
$\mu hom_{\widehat{\chi}}(F_1,\dots,F_\ell;G)$ in the same way, that is,
\begin{equation}
\mu hom_{\widehat{\chi}}(F_1,\dots,F_\ell;G) =
i_{\widehat{\chi}}^{-1}\mu_{\widehat{\chi}}(\rh(p_2^{-1}(F_1 \boxtimes \cdots \boxtimes F_\ell),\,p_1^!Ri_{\Delta*}G)).
\end{equation}
This functor is compatible with $\mu hom_{\widehat{\chi}}^{sa}$ because we have
$$
\mu hom_{\widehat{\chi}}(F_1,\dots,F_\ell;G) \simeq \rho^{-1}\mu hom_{\widehat{\chi}}(R\rho_*F_1,\dots,R\rho_*F_\ell;R\rho_*G).
$$
This is a consequence of the following facts (see \cite{Pr08} and \cite{Pr13} for details):
\begin{itemize}
\item $R\rho_*$ commutes with $p_1^!$ and $Ri_{\Delta*}$,
\item $\rho^{-1}$ commutes with $p_2^{-1}$, $\boxtimes$ and the Fourier-Sato transform,
\item $R\rho_*\rh(\rho^{-1}(\cdot),\cdot) \simeq \rh(\cdot,R\rho_*(\cdot))$ and $\rho^{-1}R\rho_* \simeq \id$,
\item the compatibility $\rho^{-1}\nu^{sa}_\chi R\rho_* \simeq \nu_\chi$ following from the stalk formula of \cite{HP13}.
\end{itemize}
Hence, the subsequent formulas for $\mu hom_{\widehat{\chi}}^{sa}$ in this subsection
also hold for $\mu hom_{\widehat{\chi}}$.
\end{oss}

Let us establish several properties of the functor, which are needed to define an action
of a multi-microlocal operator.

\begin{lem}{\label{lem:fundamental_restriction_ok}}
Let $\chi = \{M_1,\cdots, M_\ell\}$ be a family of closed subsets of forest type in $X$, and let $Z$ be a closed subanalytic subset in $X$ and 
$F \in D^b(\CC_{X_{sa}})$. Set $M = M_1 \cap \cdots \cap M_\ell$.
Assume $F_{X \setminus Z} \simeq 0$. Then we have
$$
Ri_{*} i^{-1} \mu^{sa}_\chi(F) \simeq \mu^{sa}_\chi(F),
$$
where $i: (M \cap Z) \times_M S^*_\chi \hookrightarrow S^*_\chi$ is a closed embedding.
\end{lem}
\begin{proof}
	Let $\tau_\chi: S_\chi \to M$ be the canonical projection.
It suffices to show
$$
R\Gamma_{\tau^{-1}_\chi(M \setminus Z)} \nu^{sa}_{\chi}(F) \simeq 0.
$$
Let $\tilde{p}:\widetilde{X}_\chi \to X$ be the multi-normal deformation of $X$
along $\chi$ and $\Omega = \{t_1 > 0, \cdots, t_\ell > 0\} \subset \widetilde{X}_\chi$.
Set $U = X \setminus Z$.
Since $\tilde{U}:=\tilde{p}^{-1}(U)$ is a subanalytic open neighborhood
of $\tau^{-1}_\chi(M \setminus Z)$ in $\widetilde{X}_\chi$, for
any subanalytic open subset $V \subset \tau^{-1}_\chi(M \setminus Z)$
and any subanalytic open neighborhood of $\tilde{V} \subset \widetilde{X}_\chi$ of 
$V$, the open subset $\tilde{V} \cap \tilde{U}$ is still an open neighborhood of $V$ and
we have
$$
\begin{aligned}
&R\Gamma(\tilde{V} \cap \tilde{U};\, R\Gamma_\Omega(\tilde{p}^{-1}F)) \simeq
R\Gamma(\tilde{V};\, R\Gamma_{\Omega \cap \tilde{U}}(\tilde{p}^{-1}F))  \\
&\qquad \simeq R\Gamma(\tilde{V};\, R\Gamma_{\Omega \cap \tilde{U}}((\tilde{p}^{-1}F)_{\tilde{U}})) \simeq
R\Gamma(\tilde{V};\, R\Gamma_{\Omega \cap \tilde{U}}(\tilde{p}^{-1}(F_U)))
=0.
\end{aligned}
$$
This completes the proof.
\end{proof}
\begin{lem}{\label{lem:restriction_ok}}
	Let $F_1,\dots,F_\ell$ and $G$ be in $D^b(\CC_{X_{sa}})$, and set
$$
\mathscr{M}^{sa}(F_1,\cdots,F_\ell;\,G) = \mu^{sa}_{\widehat{\chi}}(\rh(p_2^{-1}(F_1 \boxtimes \cdots \boxtimes F_\ell),\,p_1^!Ri_{\Delta*}G)).
$$
Then we have
\begin{equation}{\label{eq:diagonal_support}}
	Ri_{\widehat{\chi} *}i_{\widehat{\chi}}^{-1} \mathscr{M}^{sa}(F_1,\cdots,F_\ell;\,G) 
	\simeq \mathscr{M}^{sa}(F_1,\cdots,F_\ell;\,G).
\end{equation}
\end{lem}
\begin{proof}
Consider the following Cartesian diagram:
$$
\begin{matrix}
	X^\ell \times X^\ell & \xleftarrow{\,\,\tilde{i}_\Delta = i_\Delta \times \mathrm{id}\,\,}
& X \times X^\ell\\
\,\,\,\,\,\downarrow p_1 & & \,\,\,\,\,\downarrow \tilde{p}_1 \\
X^\ell & \xleftarrow{\,\,i_\Delta\,\,} 
& X \\
\end{matrix}
$$
Then we have
$$
\begin{aligned}
	\mathscr{M}^{sa}(F_1,\cdots,F_\ell;\,G)
&\simeq
\mu^{sa}_{\widehat{\chi}}(\rh(p_2^{-1}(F_1 \boxtimes \cdots \boxtimes F_\ell),\,
R\tilde{i}_{\Delta*}\tilde{p}_1^! G)) \\
&\simeq
\mu^{sa}_{\widehat{\chi}}(R\tilde{i}_{\Delta *}
\rh(\tilde{i}^{-1}_{\Delta}p_2^{-1}(F_1 \boxtimes \cdots \boxtimes F_\ell),\,
\tilde{p}_1^! G)).
\end{aligned}
$$
Hereafter we regard $X$ and $(\Delta_X)^\ell$ as subspaces in $X^\ell \times X^\ell$
by the diagonal embeddings
$$
X \hookrightarrow (\Delta_X)^\ell \hookrightarrow X^\ell \times X^\ell.
$$
Since
$
(\Delta_X)^\ell \,\,\bigcap\,\, \tilde{i}_{\Delta}(X \times X^\ell) = X
$
holds, the claims follows from the previous lemma.
\end{proof}

\begin{oss}
By taking the above lemma into account, 
we can regard $\mathscr{M}^{sa}(F_1,\cdots,F_\ell;\,G)$ as
an object on $S^*_{\chi,X}$. Further, we often omit $i^{-1}_{\widehat{\chi}}$
in the definition of $\mu hom^{sa}_{\widehat{\chi}}(F_1,\dots,F_\ell;G)$.
\end{oss}

\

Set 
$\widehat{\chi}_{123} = 
\widehat{\chi}(X\times Y \times Z, (\Lambda,\preceq))$ and
$\widehat{\chi}_{12} = \widehat{\chi}(X\times Y, (\Lambda,\preceq))$.
We also set $\widehat{\chi}_{23}$ and $\widehat{\chi}_{13}$ 
in the same way. By noticing 
$$
\begin{aligned}
S^*_{\widehat{\chi}_{123},\Delta} &=
T^*_{\Delta_{X\times Y \times Z}}(X\times Y \times Z)^2 
\underset{X \times Y \times Z}{\times} \cdots
\underset{X \times Y \times Z}{\times} 
T^*_{\Delta_{X\times Y \times Z}}(X\times Y \times Z)^2  \\
&=
T^*(X \times Y\times Z) \underset{X \times Y \times Z}{\times} \cdots
\underset{X \times Y \times Z}{\times} T^*(X \times Y \times Z), \\
S^*_{\widehat{\chi}_{12},\Delta} &=
T^*_{\Delta_{X\times Y}}(X\times Y)^2 
\underset{X \times Y}{\times} \cdots
\underset{X \times Y}{\times} 
T^*_{\Delta_{X\times Y}}(X\times Y)^2  \\
&=
T^*(X \times Y) \underset{X \times Y}{\times} \cdots
\underset{X \times Y}{\times} T^*(X \times Y),
\end{aligned}
$$
and similar expressions for $S^*_{\widehat{\chi}_{ij},\Delta}$ ($i < j \in \{1,2,3\}$),
we have the canonical projections 
$$
q_{13}: S^*_{\widehat{\chi}_{123}, \Delta}  \to
S^*_{\widehat{\chi}_{13}, \Delta},\quad
q_{12}: S^*_{\widehat{\chi}_{123}, \Delta}  \to
S^*_{\widehat{\chi}_{12}, \Delta},\quad
q_{23}: S^*_{\widehat{\chi}_{123}, \Delta}  \to
S^*_{\widehat{\chi}_{23}, \Delta}.
$$
Let us denote by $q_{ij}^a$ the composition of the projection $q_{ij}$ and the antipodal map on the $j$-th factor.

\begin{prop}{\label{prop:composition_muhom}} Let $F_1^1,\dots,F_\ell^1,G^1 \in D^b(\CC_{(X \times Y)_{sa}})$ and let $F_1^2,\dots,F_\ell^2,G^2 \in D^b(\CC_{(Y \times Z)_{sa}})$. There is a canonical morphism
\begin{eqnarray*}
& R(q^a_{13})_{\csapd}\left((q_{12}^{a})^{-1}\mu hom^{sa}_{\widehat{\chi}_{12}}(F_1^1,\dots,F_\ell^1;G^1) \otimes (q_{23}^{a})^{-1}\mu hom^{sa}_{\widehat{\chi}_{23}}(F_1^2,\dots,F_\ell^2;G^2)\right) & \\
	& \to \mu hom^{sa}_{\widehat{\chi}_{13}}(F_1^1 \circ F_1^2,\dots,F_\ell^1 \circ F_\ell^2;G^1 \circ G^2). &
\end{eqnarray*}
\end{prop}
\begin{proof}
The result follows from the previous proposition and the following canonical morphisms, for a general $f:X \to Y$ and $f^\ell(x_1,\dots,x_\ell)=(f(x_1),\dots,f(x_\ell))$:
$$
(f^{\ell})^{-1}(H_1 \boxtimes \cdots \boxtimes H_\ell) \simeq \imin f H_1 \boxtimes \cdots \boxtimes \imin f H_\ell,
$$
$$
Rf^\ell_!(H_1 \boxtimes \cdots \boxtimes H_\ell) \gets Rf_!H_1 \boxtimes \cdots \boxtimes Rf_!H_\ell,
$$
$$
(f^{\ell})^{-1}Ri_{\Delta*}H \to Ri_{\Delta*}\imin fH,
$$
$$
Rf^\ell_!Ri_{\Delta*}H \simeq Ri_{\Delta*}Rf_!H.
$$
Finally, by taking the commutative diagram (4.4.15) in p.213 \cite{KS90} into account,
we obtain the desired formula.
This completes the proof.
\end{proof}

Let $\chi=\{M_1,\dots,M_\ell\}$ be a family of closed submanifolds of forest type.
Since
$$
S^*_\chi = T^*_{M_1}\iota(M_1) \xtimes \cdots \xtimes T^*_{M_\ell}\iota(M_\ell) 
$$
holds by Proposition 1.5 \cite{HP13}, 
we have the surjective morphism
$$
\omega: T^*_{M_1}X \xtimes \cdots \xtimes T^*_{M_\ell}X
\longrightarrow
S^*_\chi = T^*_{M_1}\iota(M_1) \xtimes \cdots \xtimes T^*_{M_\ell}\iota(M_\ell). 
$$
Recall that we set $\widehat{\chi} = \widehat{\chi}(X, (\Lambda, \preceq))$.
The following theorem says that a usual multi-microlocalization 
$\mu^{sa}_\chi$ along $\chi$ whose induced type is $(\Lambda, \preceq)$ 
can be recovered from $\mu hom^{sa}_{\widehat{\chi}}$:

\begin{teo}{\label{thm:muhom_and_mu}} 
Assume the induced type of  $\chi=\{M_1,\dots,M_\ell\}$ is the same as $(\Lambda, \preceq)$.
Then, for any $G \in D^b(\CC_{X_{sa}})$ we have
$$
i_* \omega^{-1}\mu^{sa}_\chi(G) \simeq \mu hom^{sa}_{\widehat{\chi}}(\CC_{M_1},\dots,\CC_{M_\ell};G),
$$
where $i$ denotes the closed embedding
$$
S^*_\chi = T^*_{M_1}X \xtimes \cdots \xtimes T^*_{M_\ell}X
\hookrightarrow 
T^*X \xtimes \cdots \xtimes T^*X = S^*_{\widehat{\chi},X}.
$$
\end{teo}
\begin{proof}
Let us consider the following diagrams
$$
X \xrightarrow{\,\,i_\Delta\,\,} X^\ell \xleftarrow{\,\,p_1^\ell\,\,} 
(X\times M_j)_j := (X \times M_1)\times \cdots \times (X \times M_\ell) \xrightarrow{\,\,f_2^\ell\,\,} 
(X^2)^\ell,
$$
where $i_\Delta$ and $f_2^\ell$ are closed embeddings and $p_1^\ell$ is the canonical
projection. 

The family $\widehat{\chi}_{(X \times M_j)_j}=\{\widehat{M}_j,j=1,\dots,\ell\}$ is defined as follows: $\widehat{M}_j=Z_{j1} \times \cdots \times Z_{j\ell}$ where
$Z_{jk}=X \times M_k$ if $j \npreceq k$, $\Delta_{M_k}$ if $j \preceq k$. 
The family $\widehat{\chi}_{X^\ell}=\{\widehat{N}_j,j=1,\dots,\ell\}$ is defined as follows: $\widehat{N}_j=Z_{j1} \times \cdots \times Z_{j\ell}$ where
$Z_{jk}=X$ if $j \npreceq k$, $M_k$ if $j \preceq k$.
Then we have 
$$
\times_M T^*_{M_j}\iota(M_j) 
\xleftarrow{\,\,{}^ti_\Delta'\,\,}
M \underset{X^\ell}{\times} (\times_j T^*_{M_j}X) \xrightarrow{\,\,i_{\Delta\pi}\,\,} 
\times_j T^*_{M_j}X.
$$
Here 
$\times_M T^*_{M_j}\iota(M_j)$ denotes $T^*_{M_1}\iota(M_1) \times_X \cdots \times_X
T^*_{M_\ell}\iota(M_\ell)$ and
$\times_j T^*_{M_j}X$ does $T^*_{M_1}X \times \cdots \times T^*_{M_\ell}X$.
The map $p^\ell_1$ induces
$$
\times_j T^*_{M_j}X
\xleftarrow{\,\,{p^\ell_{1 \pi}}\,\,}
(\times_j \Delta_{M_j}) \underset{X^\ell}{\times} (\times_j T^*_{M_j}X)
\xrightarrow{\,\,{}^tp_1^\ell{}'\,\,} 
\times_j T^*_{M_j} (X\times M_j),
$$
where $p^\ell_{1 \pi}$ is isomorphic and ${}^tp_1^\ell{}'$ is a closed embedding.
Finally the map $f_2^\ell$ induces
$$
\times_j T^*_{M_j} (X\times M_j)
\xleftarrow{\,\,{}^tf_2^\ell{}'\,\,} 
(\times_j M_j) \underset{X^\ell}{\times} (\times_j T^*X)
\xrightarrow{\,\,f^\ell_{2\pi}\,\,} \times_j T^*X,
$$
where $f^\ell_{2\pi}$ is a closed embedding and ${}^tf_2^\ell{}'$ is isomorphic.
Then, by Propositions \ref{prop:direct_image} and \ref{prop:inverse_image}, we have the following isomorphisms
\begin{eqnarray*}
\mu hom^{sa}_{\widehat{\chi}}(\CC_{M_1},\dots,\CC_{M_\ell};G) & 
= & i_{\widehat{\chi}}^{-1}\mu^{sa}_{\widehat{\chi}}\rh(p_2^{-1}(\CC_{M_1} \boxtimes \cdots \boxtimes \CC_{M_\ell}),p_1^!Ri_{\Delta*}G) \\
& \simeq & i_{\widehat{\chi}}^{-1}\mu^{sa}_{\widehat{\chi}}(R(f^\ell_2)_*(f_2^{\ell})^{!}p_1^!Ri_{\Delta*}G) \\
& \simeq 
& i_{\widehat{\chi}}^{-1}R(f^\ell_{2\pi})_{*} ({}^tf_2^\ell{}')^{-1}\mu^{sa}_{\widehat{\chi}_{(X \times M_j)_j}}(p_1^{\ell!}Ri_{\Delta*}G) \\
& \simeq & i_{\widehat{\chi}}^{-1}R(f^\ell_{2\pi})_* ({}^tf_2^\ell{}')^{-1}R({}^tp_1^\ell{}')_{*} \mu^{sa}_{\widehat{\chi}_{X^\ell}}(Ri_{\Delta*}G) \\
& \simeq & i_{\widehat{\chi}}^{-1}R(f_{2\pi}^{\ell})_* ({}^tf_2^\ell{}')^{-1}R({}^tp_1^\ell{}')_*
R(i_{\Delta \pi})_* ({}^ti_{\Delta}{}')^{-1}\mu^{sa}_{\chi}(G).
\end{eqnarray*}
We easily see
$$
i_{\widehat{\chi}}^{-1}R(f^\ell_{2\pi})_* ({}^tf_2^\ell{}')^{-1}R({}^tp_{1}^\ell{}')_* R(i_{\Delta\pi})_* 
({}^ti_{\Delta}{}')^{-1}
\simeq 
Ri_*\omega^{-1} \simeq i_*\omega^{-1},
$$
which completes the proof.
\end{proof}

\section{Multi-microlocal operators}{\label{sec:multil-micro-op}}
Let $X$ be an $n$-dimensional complex manifold.
Let $(\Lambda, \preceq)$ be a ``type'' and $\Theta$ a stable subset of $\Lambda$.
Through the note,
$\widehat{\chi}$ designates the universal family $\widehat{\chi}(X, (\Lambda,\preceq))$
and $\widehat{\chi}_2$ also designates the universal family $\widehat{\chi}(X^2, (\Lambda,\preceq))$, where $X^2 = X \times X$. 

Furthermore, we assume that there exists 
a compatible family $\varpi := \{\varpi_k: X \to S_k\}_{k \in \Theta}$ 
of submersive morphisms of complex manifolds with respect to $(\Theta, \preceq|_\Theta)$ 
such that, for any point $\{s_k\}_{k \in \Theta}$ of $\varpi$,
a family
$
\{\varpi_k^{-1}(s_k)\}_{k \in \Theta}
$
of closed complex submanifolds in $X$ becomes of forest type in $X$ 
and its type is $(\Theta, \preceq|_{\Theta})$.
See Appendix for the notion of  a compatible family of submersive morphisms.

\begin{oss}
Locally such a compatible family always exists. 
Let us take a system $x = (x^{(0)};\,\{x^{(i)}\}_{i \in \Theta})$ of coordinates of $X$ and subsets $I_k \subset \Theta$ ($k\in \Theta$) 
as those in Proposition 1.2 \cite{HP13}.
Then $\{\varpi_k\}_{k \in \Theta}$ defined by
$$
\varpi_k(x^{(0)};\,\{x^{(j)}\}_{j \in \Theta}) = (x^{(i)})_{i \in I_k}\qquad (k \in \Theta),
$$
is a compatible family of submersive morphisms.
\end{oss}

Let us define a closed subset $\Delta^{\varpi_k}$ in $X^2$  by
\begin{equation}{\label{eq:delta_varpi}}
\Delta^{\varpi_k} =
\left\{
\begin{array}{ll}
\Delta_X \qquad & \text{if $k \notin \Theta$}, \\
\\
\{(x,y) \in X^2;\, 
\varpi_k(x) = \varpi_k(y)\} \quad&\text{if $k \in \Theta$}.
\end{array}
\right.
\end{equation}
For $k \in \Theta$, since $\varpi_k: X \to S_k$ is a submersion,
we have the canonical embedding
$$
X \underset{S_k}{\times} T^*{S_k} \,\hookrightarrow\, T^*X,
$$
which is the dual vector bundles of
$
X \underset{S_k}{\times} T{S_k} \,\xleftarrow{\,\varpi_k'\,}\, TX
$.
Hereafter we regard 
$X \underset{S_k}{\times} T^*{S_k}$ as a subbundle in $T^*X$.
Furthermore, the canonical morphism of vector bundles
$$
X \underset{X^2}{\times}T_{\Delta^{\varpi_k}}X^2 \ni (x,x;\, v_1, v_2) \mapsto
(x;\, \varpi'_{k,x}(v_1) - \varpi'_{k,x}(v_2)) \in X \times_{S_k} TS_k
$$
is clearly isomorphic. Hence we have obtained the isomorphism
\begin{equation}{\label{eq:iso_S_k_cotan}}
X \underset{X^2}{\times}T^{\,*}_{\Delta^{\varpi_k}}X^2 \simeq
X \times_{S_k} T^*S_k\qquad (k \in \Theta).
\end{equation}
We set
$$
S^*_{\varpi,k} = X \underset{X^2}{\times}T^{\,*}_{\Delta^{\varpi_k}}X^2 =
\begin{cases}
T^*X \qquad & (k \notin \Theta), \\
X \underset{S_k}{\times} T^*S_k \qquad & (k \in \Theta).
\end{cases}
$$
Then we define the vector bundle over $X^\ell$ by
$$
S^*_{\varpi} := 
S^*_{\varpi,1} \times 
S^*_{\varpi,2} \times \cdots \times
S^*_{\varpi,\ell}
\quad \subset  \quad S^*_{\widehat{\chi}}
$$
and the one over $X$ by
$$
S^*_{\varpi,X} := X \underset{X^\ell}{\times} S^*_{\varpi} =
S^*_{\varpi,1} \xtimes 
S^*_{\varpi,2} \xtimes \cdots \xtimes
S^*_{\varpi,\ell}
\quad \subset  \quad
S^*_{\widehat{\chi},X}.
$$
Recall $n = \mathrm{dim}_{\mathbb{C}} X$.
\begin{df}{\label{df:multi-microlocal-operators}}
The sheaf $\mathscr{E}^{\varpi,\mathbb{R}}_{\hat{\chi}}$
of multi-microlocal operators on $S^*_{\hat{\chi},X}$ is defined by
$$
\mathscr{E}^{\varpi,\mathbb{R}}_{\hat{\chi}} :=
H^n\left(i^{-1}_{\widehat{\chi}_{21}} \muhom_{\widehat{\chi}_2}
(\mathbb{C}_{\Delta^{\varpi_1}},\mathbb{C}_{\Delta^{\varpi_2}},\dots,
\mathbb{C}_{\Delta^{\varpi_\ell}};\,
\OO^{(0,n)}_{X^2})\right).
$$
The sheaf $\mathscr{E}^{\varpi,f,\mathbb{R}}_{\hat{\chi}}$
of tempered multi-microlocal operators on $S^*_{\hat{\chi},X}$ is defined by
$$
\mathscr{E}^{\varpi,f,\mathbb{R}}_{\hat{\chi}} :=
H^n\left(\rho^{-1} i^{-1}_{\widehat{\chi}_{21}} \muhom_{\widehat{\chi}_2}^{sa}(\mathbb{C}_{\Delta^{\varpi_1}},\mathbb{C}_{\Delta^{\varpi_2}},\dots,\mathbb{C}_{\Delta^{\varpi_\ell}};\,
\OO^{t\,(0,n)}_{X^2})\right).
$$
Here the canonical morphism 
$i_{\widehat{\chi}_{21}}: 
S^*_{\widehat{\chi},X} \to S^*_{\widehat{\chi}_2,X^2}$
is defined by the closed embedding
$$
T^*X \xtimes \cdots \xtimes T^*X
=
T_{\Delta_X}^*X^2 \xtimes \cdots \xtimes T_{\Delta_X}^*X^2
\hookrightarrow 
T^*X^2 \underset{X^2}{\times} \cdots \underset{X^2}{\times} T^*X^2
$$
with
$$
(x;\zeta_1,\cdots,\zeta_\ell) \mapsto
(x,x;\zeta_1,-\zeta_1, \cdots, \zeta_\ell, - \zeta_\ell).
$$
\end{df}

\

Let $G \in D^b(\CC_{(X\times X)_{sa}})$ and set
$$
\mathscr{M}^{sa}(G) :=
\mu^{sa}_{\widehat{\chi}_2}
(\rh(p_2^{-1}(\mathbb{C}_{\Delta^{\varpi_1}} \boxtimes \cdots \boxtimes 
	\mathbb{C}_{\Delta^{\varpi_\ell}}),
\,p_1^!Ri_{\Delta*}G)).
$$
Note that we define 
$\mathscr{M}(G)$ for $G \in D^b(\CC_{X\times X})$ in the same way.
Let $\widehat{\chi}_Y = \{Y_1,\cdots,Y_\ell\}$ 
be a family of closed submanifolds in $(X^2)^\ell$ defined by
\begin{equation}{\label{eq:chi_G}}
Y_k = Y_{k,1} \times Y_{k,2} \times \cdots \times Y_{k,\ell}\qquad
(k=1,2,\cdots,\ell),
\end{equation}
where
$$
Y_{k,j} :=
\begin{cases}
\Delta^{\varpi_j} \quad & (k \preceq j), \\
X^2 \quad & (\text{otherwise}).
\end{cases}
$$
By the same argument as in the proof of Theorem \ref{thm:muhom_and_mu}, we have
\begin{equation}\label{eq:muhom_micro_alt}
	\mathscr{M}^{sa}(G) \simeq
Rj_*\mu^{sa}_{\widehat{\chi}_Y} (Ri_{\Delta *}G),
\end{equation}
where $j$ is the closed embedding
$$
j: T^*_{\Delta^{\varpi_1}}X^2 \times \cdots \times T^*_{\Delta^{\varpi_\ell}}X^2 \to
T^*X^2 \times \cdots \times T^*X^2.
$$
Since $i_\Delta(X^2) \bigcap (\Delta^{\varpi_1} \times \cdots \times \Delta^{\varpi_\ell}) 
= i_\Delta(\Delta_X) \simeq X$
holds,
by Lemma \ref{lem:fundamental_restriction_ok} we get
\begin{equation}
Rj_{\varpi}{}_*j_{\varpi}^{-1}
R\mu_{\widehat{\chi}_Y}^{sa} (Ri_{\Delta *}G)
\simeq
R\mu_{\widehat{\chi}_Y}^{sa} (Ri_{\Delta *}G),
\end{equation}
where $j_{\varpi}$ is the closed embedding
$$
S^*_{\varpi,X} = S^*_{\varpi_1} \underset{X}{\times} \cdots \underset{X}{\times} S^*_{\varpi_\ell}
\hookrightarrow
T^*_{\Delta^{\varpi_1}}X^2 \times \cdots \times T^*_{\Delta^{\varpi_\ell}}X^2.
$$
Hence we can regard $R\mu_{\widehat{\chi}_Y}^{sa} (Ri_{\Delta *}G)$
as an object on $S^*_{\varpi,X}$.

\

Let us consider the following commutative diagram,
where all the morphisms are injective and all the squares are Cartesian:
$$
\begin{matrix}
S^*_{\varpi_1} \times \cdots \times  S^*_{\varpi_\ell}
& \xrightarrow{\,\,\tilde{i}_{\widehat{\chi}_{1 \varpi}}\,\,}
&T^*_{\Delta_X}X^2 \times \cdots \times  T^*_{\Delta_X}X^2 
&\xrightarrow{\,\,\tilde{i}_{\widehat{\chi}_{21}}\,\,}
&T^*X^2 \times \cdots \times  T^*X^2  \\
\\
\,\,\,\,\,\uparrow i_{\widehat{\chi}_\varpi} 
&& \,\,\,\,\,\uparrow i_{\widehat{\chi}}\,\,\,
&& \,\,\,\,\,\uparrow i_{\widehat{\chi}_2}
\\
S^*_{\varpi,X} = S^*_{\varpi_1} \underset{X}{\times} \cdots \underset{X}{\times}
S^*_{\varpi_\ell} 
& \xrightarrow{\,\,{i}_{\widehat{\chi}_{1 \varpi}}\,\,}
&T^*_{\Delta_X}X^2 \underset{X}{\times} \cdots \underset{X}{\times} T^*_{\Delta_X}X^2 
&\xrightarrow{\,\,{i}_{\widehat{\chi}_{21}}\,\,}
&T^*X^2 \underset{X^2}{\times} \cdots \underset{X^2}{\times}  T^*X^2.
\end{matrix}
$$
Note that $j \circ j_{\varpi}=  
\tilde{i}_{\widehat{\chi}_{21}} \circ \tilde{i}_{\widehat{\chi}_{1\varpi}}
\circ i_{\widehat{\chi}_\varpi}$ holds.
Then it follows from Lemma \ref{lem:restriction_ok} and \eqref{eq:muhom_micro_alt} that
we have
$$
\begin{aligned}
	\mathscr{M}^{sa}(G)
&=
Ri_{\widehat{\chi}_2 *}i_{\widehat{\chi}_2}^{-1}
Rj_*\mu_{\widehat{\chi}_Y}^{sa} (Ri_{\Delta *}G) \\
&=
Ri_{\widehat{\chi}_2 *}i_{\widehat{\chi}_2}^{-1}
Rj_*
Rj_{\varpi *}j_{\varpi}^{-1}
\mu_{\widehat{\chi}_Y}^{sa} (Ri_{\Delta *}G) \\
&=
Ri_{\widehat{\chi}_2 *}i_{\widehat{\chi}_2}^{-1}
R\tilde{i}_{\widehat{\chi}_{21} *} 
R\tilde{i}_{\widehat{\chi}_{1\varpi} *}
Ri_{\widehat{\chi}_\varpi *}
j_{\varpi}^{-1}
\mu_{\widehat{\chi}_Y}^{sa} (Ri_{\Delta *}G) \\
&= Ri_{\widehat{\chi}_2 *}Ri_{\widehat{\chi}_{21} *} i_{\widehat{\chi}}^{-1}
R\tilde{i}_{\widehat{\chi}_{1\varpi} *}
Ri_{\widehat{\chi}_\varpi *}
j_{\varpi}^{-1}
\mu_{\widehat{\chi}_Y}^{sa} (Ri_{\Delta *}G) \\
&= Ri_{\widehat{\chi}_2 *}Ri_{\widehat{\chi}_{21} *} 
Ri_{\widehat{\chi}_{1\varpi} *}
j_{\varpi}^{-1}
\mu_{\widehat{\chi}_Y}^{sa} (Ri_{\Delta *}G) \\
&= Rj_{\widehat{\chi}_2 *} j_{\varpi}^{-1}
\mu_{\widehat{\chi}_Y}^{sa} (Ri_{\Delta *}G),
\end{aligned}
$$
where $j_{\widehat{\chi}_2}: S^*_{\varpi,X} \hookrightarrow S^*_{\widehat{\chi}_2}$ 
is the closed embedding. Summing up, we have obtained
$$
\mathscr{M}^{sa}(G)
= Rj_{\widehat{\chi}_2 *} j_{\varpi}^{-1}
\mu_{\widehat{\chi}_Y}^{sa} (Ri_{\Delta *}G),
$$
and hence, we get
\begin{equation}
	\mathscr{M}^{sa}(G)
= Rj_{\widehat{\chi}_2 *} j_{\widehat{\chi}_2}^{-1}
\mathscr{M}^{sa}(G)
\end{equation}
Therefore we can regard $\mathscr{M}^{sa}(G)$ as an object on $S^*_{\varpi,X}$.
In particular, we have
\begin{equation}\label{eq:estimate_S_pi}
\muhom_{\widehat{\chi}_2}^{sa}(\mathbb{C}_{\Delta^{\varpi_1}},\mathbb{C}_{\Delta^{\varpi_2}},\dots,\mathbb{C}_{\Delta^{\varpi_\ell}};\,
\OO^{t\,(0,n)}_{X^2})
\simeq Ri_{\widehat{\chi}_{1 \varpi} *}
i_{\widehat{\chi}_{1 \varpi}}^{-1}
\muhom_{\widehat{\chi}_2}^{sa}(\mathbb{C}_{\Delta^{\varpi_1}},\mathbb{C}_{\Delta^{\varpi_2}},\dots,\mathbb{C}_{\Delta^{\varpi_\ell}};\,
\OO^{t\,(0,n)}_{X^2}).
\end{equation}
Furthermore, under these identifications, we have obtained the following expressions of 
the sheaf of multi-microlocal operators.
\begin{equation}
\mathscr{E}^{\varpi,\mathbb{R}}_{\hat{\chi}} = 
H^n\left(\mathscr{M}(\OO^{(0,n)}_{X^2})\right) = 
H^n\left(\mu_{\widehat{\chi}_Y} (Ri_{\Delta *}\OO^{(0,n)}_{X^2})\right).
\end{equation}
In the same way, we have
\begin{equation}
\mathscr{E}^{\varpi,f,\mathbb{R}}_{\hat{\chi}} =
H^n\left(\rho^{-1}\mathscr{M}^{sa}(\OO^{t, (0,n)}_{X^2})\right) = 
H^n\left(\rho^{-1} 
\mu^{sa}_{\widehat{\chi}_Y} 
(Ri_{\Delta *}\OO^{t\,(0,n)}_{X^2})\right).
\end{equation}

\

Set $\Lambda' = \Theta \cup \mathrm{m}^*(\Lambda,\Theta)$. Then $(\Lambda', \preceq)$
also becomes of forest type. Recall that $\hat{\chi}'$ (resp. $\hat{\chi}_2'$) denotes
the universal family $\hat{\chi}(X,(\Lambda', \preceq))$ (resp. $\hat{\chi}(X^2,(\Lambda', \preceq))$) as usual. 
We have
$$
S^*_{\hat{\chi}',X} = \underset{k \in \Lambda',\, X}\times T^*X,
$$
and the canonical projection
$$
p_{\hat{\chi}', \hat{\chi}}: S^*_{\widehat{\chi},X}  \to S^*_{\widehat{\chi}',X}.
$$
Then the following theorem says that the indices $\Lambda \setminus \Lambda'$ are redundant
for a microlocal operator.
\begin{teo}{\label{thm:redundunt_indices}}
	Assume that $i \nmid j$ holds for any $i \ne j \in \mathrm{m}^*(\Lambda,\Theta)$.
Then we have
$$
\mathscr{E}^{\varpi,f,\mathbb{R}}_{\hat{\chi}} \simeq
p^{-1}_{\hat{\chi}',\hat{\chi}}\, \mathscr{E}^{\varpi,f,\mathbb{R}}_{\hat{\chi}'}
\quad\text{and}\quad
\mathscr{E}^{\varpi,\mathbb{R}}_{\hat{\chi}} \simeq
p^{-1}_{\hat{\chi}',\hat{\chi}}\, \mathscr{E}^{\varpi,\mathbb{R}}_{\hat{\chi}'}.
$$
\end{teo}

To show the theorem, we need the following lemma:
Let $\chi = \{N_1,\cdots,N_\ell\}$ be a family of real analytic 
closed and connected submanifolds in $X$, and
let $\chi'$ be a subfamily of $\chi$.
Recall the definitions of ``transitive type'' and ``the rank of a family'' introduced
in Definition 1.8 \cite{HP24} and after the same definition, respectively.
Note that, if $\chi$ satisfies the condition H2 and it is of transitive type,
then $S_\chi$ has a vector bundle structure over $N= N_1 \cap \cdots \cap N_\ell$
by Theorem 1.9 \cite{HP24} (see also Proposition 1.10 \cite{HP24}). 
Hence, in this case, we can define
the multi-microlocalization $\mu_\chi^{sa}(\bullet)$ along $\chi$ by 
applying the (multi-)Fourier transform to $\nu_{\chi}^{sa}(\bullet)$ in the same way
as the usual one.
\begin{lem}{\label{lem:remove_duplicate}}
Assume that $\chi$ satisfies the condition H2.
Furthermore we also assume that $\chi$ and $\chi'$ are of transitive type
and the rank of the family $\chi$ and that of $\chi'$ are the same.
Then we have $S_{\chi} = S_{\chi'}$ and 
$$
\nu^{sa}_\chi(F) \simeq \nu^{sa}_{\chi'}(F).
$$
As a consequence, we have $S^*_{\chi} = S^*_{\chi'}$ and
$$
\mu^{sa}_\chi(F) \simeq \mu^{sa}_{\chi'}(F).
$$
\end{lem}
\begin{proof}
We may assume $\chi'= \{N_1,N_2,\cdots,N_{\ell-d}\}$ for some $0 \le d \le \ell -1$. 
Set 
$$
\lambda = (\lambda_1,\cdots,\lambda_{\ell-d},\lambda_{\ell-d+1},\cdots, \lambda_\ell) 
= (\lambda',\lambda''),
$$
and let $p:\widetilde{X}_\chi \to X$ be the canonical map of the multi-normal deformation
along $\chi$.
Since the rank of the families are the same, we see $S_{\chi} = S_{\chi'}$.
As the families are of transitive type,
it follows from Theorem 1.9 \cite{HP24} that
$S_{\chi}$ and $S_{\chi'}$ are vector bundles over 
$N = N_1 \cap \cdots \cap N_{\ell -d} = N_1 \cap \cdots \cap N_\ell$. 

Let us consider the canonical map 
$
p'' : \tilde{X}_{\chi \setminus \chi'} \to X
$
of the multi-normal deformation along $\chi \setminus \chi'$ and the morphism 
$$
F \to R(p'')_*(p'')^{-1} F.
$$
Note that $p''(x,\lambda'') = p(x,1',\lambda'')$ holds, where $1'=(1,1,\cdots,1)$.
Now let us apply multi-specialization along $\chi'$ to the above morphism.
It is easy to see that
$$
T_{\chi'}p'':
S^{\tilde{X}_{\chi \setminus \chi'}}_{\chi'} \simeq X \times \mathbb{R}^d
\to 
S^{X}_{\chi'} \simeq X
$$
is also given by
$$
(x,\lambda'') \mapsto p''(x,\lambda'') = p(x, 1',\lambda'').
$$
It follows from Proposition 3.18 \cite{HP24} that we get the morphism
$$
\nu^{sa}_{\chi'}(F) \to \nu^{sa}_{\chi'}(R(p'')_*(p'')^{-1}F) \to
R(T_{\chi'}p'')_* \nu^{sa}_{\chi'}((p'')^{-1}F).
$$

Let $G$ be a sheaf on $\widehat{S}:=S^{\tilde{X}_{\chi \setminus \chi'}}_{\chi'} \simeq X \times \mathbb{R}^d$
which is multi-conic with respect to the action 
$\mu_{\widehat{S}} \simeq \mu' \times \mathrm{id}_{\mathbb{R}^d}$ 
(here $\mu'$ is the action on $X$ associated with the family $\chi'$)
and $V \subset X$ a subanalytic open subset which is also multi-conic with respect to
the action $\mu_{S_{\chi'}}\simeq \mu'$. We have
$$
\Gamma(V;\, (T_{\chi'}p'')_* G) \to 
\Gamma(V;\, (T_{\chi'}p'')_* \Gamma_{\Omega''}G) \simeq
\Gamma((T_{\chi'}p'')^{-1}(V);\, \Gamma_{\Omega''}G) \to
\Gamma(V;\, (\Gamma_{\Omega''}G)|_{X \times \{0\}}),
$$
where $\Omega'' = \{(x,\lambda'') \in 
	\widehat{S} \simeq X \times \mathbb{R}^d;\,
\lambda_{\ell-d+1} > 0,\cdots, \lambda_{\ell} > 0\}$.
Note that the third morphism is just the restriction because
the actions on $X$ associated with $\chi$ and $\chi'$ coincide,
and thus, we have 
$$
(T_{\chi'}p'')^{-1}(V) \cap \Omega'' = (V \times \mathbb{R}^d) \cap \Omega''.
$$
Therefore we have obtained the morphism of multi-conic sheaves
$$
(T_{\chi'}p'')_* G \to (\Gamma_{\Omega''}G)|_{X \times \{0\}}.
$$
Summing up we have the morphism,
$$
\nu^{sa}_{\chi'}(F) \to 
R(T_{\chi'}p'')_* \nu^{sa}_{\chi'}((p'')^{-1}F) \to
(R\Gamma_{\Omega''}\nu^{sa}_{\chi'}((p'')^{-1}F))|_{X \times \{0\}}
\to \nu^{sa}_{\chi}(F).
$$

Finally it follows from Theorem 3.8 and Corollary 2.54 in \cite{HP24} that 
the above morphism is isomorphic.
\end{proof}

Now we give the proof of the theorem.

\begin{proof}
Recall the family $\widehat{\chi}_Y$ defined by \eqref{eq:chi_G}.
We also denote by $\widehat{\chi}'_Y$ the corresponding one with respect to
$(\Lambda',\preceq)$.
Then we have
$$
\mathscr{E}^{\varpi,f,\mathbb{R}}_{\hat{\chi}} =
\rho^{-1}\mu^{sa}_{\widehat{\chi}_Y} (Ri_{\Delta *}\OO^{t\,(0,n)}_{X^2}).
$$
We define $\sigma: \Lambda \setminus \Lambda' \to \mathrm{m}^*(\Lambda,\Theta)$ as follows:
For $k \in \Lambda\setminus \Lambda'$, $\sigma(k)$ is taken so that
$\sigma(k) \in \mathrm{m}^*(\Lambda,\Theta)$ and $k \preceq \sigma(k)$ hold.
Such a $\sigma(k)$ is not unique, however it surely exists by Lemma {\ref{lem:m-star-1}}.
Then the map $j: \times_{k \in \Lambda'} X^2 \to 
\times_{k \in \Lambda} X^2$ is defined by
$$
j((x^{(0)}; \{x^{(k)}\}_{k \in \Lambda'})) = (x^{(0)}; 
\{y^{(k)}\}_{k \in \Lambda}\})
$$
with
$$
y^{(k)} =
\begin{cases}
	x^{(k)} \qquad & (k \in \Lambda'),\\
x^{(\sigma(k))} \qquad & (k \in \Lambda \setminus \Lambda')).
\end{cases}
$$
We denote by $i'_\Delta: X^2 \to \times_{k \in \Lambda'} X^2$ the diagonal
embedding. Then clearly $j \circ i'_\Delta  = i_\Delta$ holds.
Define the family of submanifolds in $\times_{k \in \Lambda'} X^2$ by
$$
\tilde{\chi}_Y = \{\tilde{Y}_1,\cdots, \tilde{Y}_\ell\} = 
\{j^{-1}(Y_1),\cdots, j^{-1}(Y_\ell)\}.
$$
We see that $\tilde{\chi}_Y$ is of transitive type and
it satisfies the condition H2.
Note that we have $\tilde{Y}_k = \iota_{\tilde{\chi}_Y}(\tilde{Y}_k)$
for $k \in \Lambda \setminus \Lambda'$.
Since the map $j$ satisfies the condition of
Proposition \ref{prop:multimicro-proper} 
and since Proposition \ref{prop:multimicro-proper} still holds in this situation 
thanks to Proposition 3.18 \cite{HP24}, we obtain the isomorphism
$$
\mu^{sa}_{\widehat{\chi}_Y} (Ri_{\Delta *}\OO^{t\,(0,n)}_{X^2}) \simeq
\tilde{p}^{-1}\mu^{sa}_{\tilde{\chi}_Y} (Ri'_{\Delta *}\OO^{t\,(0,n)}_{X^2}),
$$
where $\tilde{p}$ is the canonical projection
$$
\tilde{p}: S^*_{\widehat{\chi}_Y} = 
\underset{k \in \Lambda}{\times} T^*_{\Delta^{\varpi_k}}X^2 \longrightarrow
\underset{k \in \Lambda'}{\times} T^*_{\Delta^{\varpi_k}}X^2 = S^*_{\tilde{\chi}_Y}.
$$
Finally, by applying Lemma \ref{lem:remove_duplicate} to $\tilde{\chi}_Y$, 
we obtain the desired result.
\end{proof}

\begin{es}{\label{es:t-s-op}}
Let $(\Lambda=\{1,2\},\preceq)$ be of totally ordered type, i.e., $2 \precneq 1$ and
set $\hat{\chi} = \hat{\chi}(X, (\Lambda, \preceq))$.

We first consider the case $\Theta = \emptyset$ and $\varpi = \emptyset$ (thus
$\mathrm{m}^*(\Lambda,\Theta) = \{1\}$).
Then $\mathscr{E}^{\emptyset,\mathbb{R}}_{\hat{\chi}}$ is nothing but the sheaf
$\mathscr{E}^\mathbb{R}_X$ of the usual microlocal operators, precisely speaking, we have
$$
\mathscr{E}^{\emptyset,\mathbb{R}}_{\hat{\chi}} = \pi_1^{-1} \mathscr{E}^{\mathbb{R}}_X,
$$
where $\pi_1: T^*X \times_X T^*X \to T^*X$ is the canonical projection with respect to
the first $T^*X$.

Now let us consider the case $\Theta =\{1\}$ and $\varpi=\{\varpi_1\}$ with
a submersive map $\varpi_1$ from $X$ to some complex manifold $S_1$ (thus
$\mathrm{m}^*(\Lambda,\Theta) = \{2\}$).
Then $\mathscr{E}^{\varpi,\mathbb{R}}_{\hat{\chi}}$ is a sheaf on
$S^*_{\varpi,X} =  T^*{S_1} \times_{S_1} T^*X$,
which is called the sheaf of small second microlocal operators of totally ordered type.
\end{es}
\begin{es}
Let $(\Lambda=\{1,2\},\preceq)$ be of normal crossing type, i.e., $1 \nmid 2$, and
set $\hat{\chi} = \hat{\chi}(X, (\Lambda, \preceq))$.

Let us consider the case $\Theta = \emptyset$ and $\varpi = \emptyset$
(thus $\mathrm{m}^*(\Lambda,\Theta) = \{1,2\}$).
As opposed to totally ordered case,
the sheaf $\mathscr{E}^{\emptyset,\mathbb{R}}_{\hat{\chi}}$ is different from
the one $\mathscr{E}^\mathbb{R}_X$ of the usual microlocal operators.
In fact, a microlocal operator in $\mathscr{E}^\mathbb{R}_X$ does not act on
a multi-microfunction of normal crossing type. Note that, as we see later, 
$\mathscr{E}^{\emptyset,\mathbb{R}}_{\hat{\chi}}$ acts 
on a multi-microfunction of normal crossing type.
\end{es}

Now we study a multi-microlocal action of $\mathscr{E}^{\varpi,\mathbb{R}}_{\hat{\chi}}$.
In what follows, we extend the family $\{\varpi_k\}_{k \in \Theta}$ to 
$\{\varpi_k\}_{k \in \Lambda}$ by
$\varpi_k = \operatorname{id}_X: X \to S_k:= X$ for $k \notin \Theta$.
Let $F_k \in D^b_{\rcc}(X)$ ($k \in \Lambda$) satisfying the condition: 
\begin{itemize}
	\item[($\dagger$)]\,\, 
There exists an $\tilde{F}_k \in D^b_{\rcc}(S_k)$ such that
\begin{equation}
	F_k \simeq \varpi^{-1}_k \tilde{F}_k.
\end{equation}
\end{itemize}
Note that, if $k \notin \Theta$, then no condition is imposed on each $F_k$
because of $\varpi_k = \mathrm{id}_X$.
\begin{es}
	Let $T_k$ $(k \in \Lambda)$ be closed subanalytic subsets in $X$.
We assume the following condition:
\begin{itemize}
	\item[$(\dagger')$]\,\, 
$T_k$ is union of
fibers of $\varpi_k$ for any $k \in \Theta$, that is,
\begin{equation}
T_k = \varpi_k^{-1}(\varpi_k(T_k))	\qquad (k \in \Theta).
\end{equation}
\end{itemize}
Set $F_k = \mathbb{C}_{T_k}$ $(k \in \Lambda)$. Then each $F_k$ satisfies the 
condition $(\dagger)$.
\end{es}
Recall that $i_{\widehat{\chi}_{21}}:S^*_{\widehat{\chi},X} \hookrightarrow S^*_{\widehat{\chi}_2}$ and the closed subset $\Delta^{\varpi_k} \subset X^2$ ($k \in \Lambda$) were defined in Definition {\ref{df:multi-microlocal-operators}} and {\eqref{eq:delta_varpi}},
respectively.
\begin{teo}
Under the above situation, we have the canonical morphisms
$$
\begin{aligned}
	&\muhom_{\widehat{\chi}_2}(\mathbb{C}_{\Delta^{\varpi_1}},\mathbb{C}_{\Delta^{\varpi_2}},\dots,\mathbb{C}_{\Delta^{\varpi_\ell}};\,
\OO^{(0,n)}_{X^2})[n]
\,\,\otimes\,\,
\muhom_{\widehat{\chi}}(F_1,F_2,\dots,
F_\ell;\,
\OO_{X})  \\
&\qquad\qquad\longrightarrow
\muhom_{\widehat{\chi}}(F_1,F_2,\dots,
F_\ell;\,
\OO_{X})
\end{aligned}
$$
and
$$
\begin{aligned}
	&\muhom^{sa}_{\widehat{\chi}_2}
(\mathbb{C}_{\Delta^{\varpi_1}},\mathbb{C}_{\Delta^{\varpi_2}},\dots,
\mathbb{C}_{\Delta^{\varpi_\ell}};\,
\OO^{t\,(0,n)}_{X^2})[n]
\,\,\otimes\,\,
\muhom^{sa}_{\widehat{\chi}}(F_1,F_2,\dots,
F_\ell;\,
\OO^t_{X})  \\
&\qquad\qquad\longrightarrow
\muhom^{sa}_{\widehat{\chi}}(F_1,F_2,\dots,
F_\ell;\,
\OO^t_{X}).
\end{aligned}
$$
\end{teo}

\begin{proof}
Let $q_1$ denote the projection $X \times X \ni (x,y) \mapsto x \in X$.
Let $K_j \subset X^2$ ($j \in \Lambda$) 
be a subanalytic closed neighborhood of $\Delta_X$ such that
$q_1|_{K_j}: K_j \to X$ is proper and, if $j \in \Theta$,
the subset
$$
(\{x\} \times \varpi_j^{-1}(\varpi_j(x))) \cap K_j \,\subset\, \{x\}\times X \subset X^2
$$
is contractible for any $x \in X$.
Note that always we can construct such a $K_j$ by using a real analytic metric on
$X$ and its induced metric on each fiber of $\varpi_j$.
Since there exists $k \in \Lambda$ such that $\Delta^{\varpi_k} = \Delta_X$
(the condition 1.~of Definition \ref{df:stable_subset}),
we have $\Delta^{\varpi_1} \cap \cdots \cap \Delta^{\varpi_\ell} = \Delta_X$ 
from which we get
$$
X^{2\ell} \times X^{2\ell}\,\supset\,
\left(\left(
\overline{\left(\Delta^{\varpi_1} \times \cdots \times \Delta^{\varpi_\ell}\right)
\setminus (K_1 \times \cdots \times K_\ell)}\right) \times i_\Delta(X^2)\right)
\,\bigcap\, (\Delta_{X^2})^\ell = \emptyset.
$$
Hence it follows form Lemma \ref{lem:fundamental_restriction_ok} that
we obtain
$$
\mu hom^{sa}_{\widehat{\chi}_2}(\CC_{\Delta^{\varpi_1}},
\dots,\CC_{\Delta^{\varpi_\ell}};
\OO^{t\,(0,n)}_{X^2})
\simeq
\mu hom^{sa}_{\widehat{\chi}_2}(\CC_{\hat{G}_1},\dots,\CC_{\hat{G}_\ell};
\OO^{t\,(0,n)}_{X^2}),
$$
where we set $\hat{G}_j = \Delta^{\varpi_j} \cap K_j$ for $j \in \Lambda$.

The canonical closed embedding $\Delta_X \hookrightarrow \hat{G}_j$ 
induces the morphism
$$
\CC_{\hat{G}_j} \circ F_j \longrightarrow \CC_{\Delta_X} \circ F_j \simeq F_j.
$$
For any $x \in X$ and $j \in \Theta$, we have
$$
H^i(\CC_{\hat{G}_j} \circ F_j)_x \simeq
R^i\Gamma ((\{x\} \times \varpi_j^{-1}(\varpi_j(x))) \cap K_j; \varpi_j^{-1}\tilde{F}_j)
\simeq H^i(\tilde{F}_j)_{\varpi_j(x)} \simeq H^i(F_j)_x,
$$
where we identify 
$(\{x\} \times \varpi_j^{-1}(\varpi_j(x))) \cap K_j$ with a subset in $X$.
Hence we have 
$
\CC_{\hat{G}_k} \circ F_k
\simeq
F_k
$
for $k \in \Theta$, 
and $\CC_{\hat{G}_k} \circ F_k = \CC_{\Delta_X} \circ F_k \simeq F_k$ for $k \notin \Theta$. Moreover we have
$$
\OO_{X^2}^{t\,(0,n)} \circ \OO^t_X[n] \to \OO^t_X.
$$
Hence by noticing \eqref{eq:estimate_S_pi} and the fact
$$
S^*_{\varpi,X} \subset T^*_{\Delta_X} X^2 \times_X \cdots \times_X T^*_{\Delta_X} X^2,
$$
the result follows from Proposition \ref{prop:composition_muhom}
with $X=Y$ and $Z$ reduced to a point.
\end{proof}
By the same argument as in the above proof, we also have the following corollary:
\begin{cor}{\label{cor:ring_struct}}
We have the canonical morphisms
$$
\begin{aligned}
& \muhom_{\widehat{\chi}_2}(\mathbb{C}_{\Delta^{\varpi_1}},
\mathbb{C}_{\Delta^{\varpi_2}},\dots,\mathbb{C}_{\Delta^{\varpi_\ell}};\,
\OO^{(0,n)}_{X^2})[n]
\,\otimes\,
\muhom_{\widehat{\chi}_2}
(\mathbb{C}_{\Delta^{\varpi_1}},\mathbb{C}_{\Delta^{\varpi_2}},\dots,
\mathbb{C}_{\Delta^{\varpi_\ell}};\,
\OO^{(0,n)}_{X^2})
\\
&\qquad\qquad \to \muhom_{\widehat{\chi}_2}
(\mathbb{C}_{\Delta^{\varpi_1}},\mathbb{C}_{\Delta^{\varpi_2}},\dots,
	\mathbb{C}_{\Delta^{\varpi_\ell}};\,
\OO^{(0,n)}_{X^2})
\end{aligned}
$$
and
$$
\begin{aligned}
&\muhom^{sa}_{\widehat{\chi}_2}
(\mathbb{C}_{\Delta^{\varpi_1}},\mathbb{C}_{\Delta^{\varpi_2}},\dots,
\mathbb{C}_{\Delta^{\varpi_\ell}};\,
\OO^{t\,(0,n)}_{X^2})[n]
\,\otimes\,
\muhom^{sa}_{\widehat{\chi}_2}
(\mathbb{C}_{\Delta^{\varpi_1}},\mathbb{C}_{\Delta^{\varpi_2}},\dots,
\mathbb{C}_{\Delta^{\varpi_\ell}};\,
\OO^{t\,(0,n)}_{X^2})
\\
&\qquad\qquad \to  
\muhom^{sa}_{\widehat{\chi}_2}
(\mathbb{C}_{\Delta^{\varpi_1}},\mathbb{C}_{\Delta^{\varpi_2}},\dots,
\mathbb{C}_{\Delta^{\varpi_\ell}};\,
\OO^{t\,(0,n)}_{X^2}).
\end{aligned}
$$
\end{cor}
As an immediate consequence of the above corollary, 
$\mathscr{E}^{\varpi,f,\mathbb{R}}_{\hat{\chi}}$ and
$\mathscr{E}^{\varpi,\mathbb{R}}_{\hat{\chi}}$ become sheaves of rings.

\

Let $\chi=\{M_1,\cdots,M_\ell\}$ be a family of closed real analytic submanifolds
of forest type in $X$ whose induced type is $(\Lambda,\preceq)$. Assume the condition $(\dagger)$ for $\chi$.
Then, by Theorem \ref{thm:muhom_and_mu}, we have
$$
\muhom^{sa}_{\widehat{\chi}}(\mathbb{C}_{M_1},\mathbb{C}_{M_2},\dots,
\mathbb{C}_{M_\ell};\, \OO^t_{X})
\simeq \omega^{-1}\mu^{sa}_\chi(\OO^t_X).
$$
Hence, if a family $\chi$ of the forest type $(\Lambda, \preceq)$ satisfies the condition $(\dagger)$, then it follows from the above theorems that we have the canonical morphism
$$
\mathscr{E}^{\varpi,f,\mathbb{R}}_{\hat{\chi}} \otimes 
\rho^{-1}\omega^{-1}\mu^{sa}_\chi(\OO^t_X) \longrightarrow
\rho^{-1}\omega^{-1}\mu^{sa}_\chi(\OO^t_X),
$$
which defines the action of $\mathscr{E}^{\varpi,f,\mathbb{R}}_{\hat{\chi}}$ on
tempered multi-microfunctions along $\chi$. In the same way, we have
$$
\mathscr{E}^{\varpi,\mathbb{R}}_{\hat{\chi}} \otimes 
\omega^{-1}\mu_\chi(\OO_X) \longrightarrow
\omega^{-1}\mu_\chi(\OO_X).
$$

\begin{es}{\label{es:t-s-ac}}
Let $(\Lambda=\{1,2\}, \preceq)$ be of totally ordered type, i.e., $2 \precneq 1$ and
let $\chi = \{M_1, M_2\}$ be a family of real analytic closed submanifolds 
of forest type in $X$
whose type is $(\Lambda, \preceq)$. 

Let $\Theta = \emptyset$. Then
a microlocal operator $P \in \mathscr{E}^{\emptyset,\mathbb{R}}_{\hat{\chi}}$ 
acts on the cohomology groups $H^k(\mu_\chi(\OO_X))$.

Let us consider the case $\Theta =\{1\}$ and $\varpi=\{\varpi_1\}$ with
a submersive map $\varpi_1$ from $X$ to some complex manifold $S_1$.
If $M_1$ satisfies the condition $(\dagger)'$, i.e., there exists a real analytic
submanifold $N \subset S_1$ such that $M_1 = \varpi_1^{-1}(N)$, then
a small second microlocal operator $P \in \mathscr{E}^{\varpi,\mathbb{R}}_{\hat{\chi}}$ 
of totally ordered type acts on the cohomology groups $H^k(\mu_\chi(\OO_X))$.
\end{es}

\appendix
\begin{section}{Appendix}
\subsection{A compatible family of submersive morphisms}


Let us consider a family $\{\varpi_k:X \to S_k\}_{k \in \Theta}$ of submersive morphisms,
where $S_k$ is a real analytic manifold and $\varpi_k:X \to S_k$ is a submersion.

We say that a family $\{\varpi_k\}$ is compatible with respect to the type $(\Theta,\preceq)$ if
there exists a family of morphism $\{\iota_{k,j}:S_j \to S_k\}_{j \preceq k}$ satisfying
the following conditions:
\begin{enumerate}
\item We have, for $j \preceq k$,
\begin{equation}{\label{eq:ap-xxx-1}}
\varpi_k = \iota_{k,j} \circ \pi_j.
\end{equation}
\item 
$\iota_{\ell,j} = \iota_{\ell,k}\circ\iota_{k,j}$ holds for 
$j \preceq k \preceq \ell$.
\end{enumerate}
Note that it follows from the above condition 1.~that $\iota_{k,k} = \operatorname{id}$ and $\iota_{k,j}$ becomes a submersion.
We say that $s=\{s_k\}_{k \in \Theta}$ with $s_k \in S_k$ $(k \in \Theta)$ is a point of the family $\{\varpi_k\}$ if
$
s_k = \iota_{k,j}(s_j)
$
holds for $j \preceq k$. 
Note that, for any point $x$, $\{\varpi_k(x)\}_{k \in \Theta}$ becomes a point of $\{\varpi_k\}$ and
denote it by $\varpi(x)$.
We also says that $x \in X$ is a point over $s = \{s_k\}$ if
$
s = \varpi(x)
$
holds.

The following theorem says that, for a compatible submersive system,
simultaneously each $\varpi_k$ is trivialized with the fiber $M_k = \varpi^{-1}_k(s_k)$
(i.e., $\varpi_k: X = M_k \times S_k \to S_k$) for some $s = \{s_k\}$.
\begin{teo}{\label{thm:c-s-m-coordinate}}
Let $\{\varpi_k:X \to S_k\}_{k \in \Theta}$ be a compatible family of 
submersive morphisms with respect to $(\Theta,\preceq)$ and $\{s_k\}$ its point.
Assume that $\{M_k := \varpi^{-1}_k(s_k)\}_{k \in \Theta}$ becomes a family of real analytic closed submanifolds of the forest type $(\Lambda,\preceq)$ in $X$.
Then, for any point $q \in M = \cap_k M_k$, there exist an
open neighborhood $U$ of $q$, a system $x = (x^{(0)}; \{x^{(k)}\}_{k \in \Theta})$ of 
coordinates in $U$ 
and subsets $I_k \subset \Theta$ ($k \in \Theta$) such that
by suitable coordinate transformations of $X$ and $S_k$ $(k \in \Theta)$ we have
the followings:
\begin{enumerate}
\item $q = 0$ and $s_k = 0$ ($k \in \Theta$).
\item Each $\varpi_k$ becomes a mapping 
$$ X \ni (x^{(0)}; \{x^{(j)}\}_{j \in \Theta})
\mapsto (x^{(j)})_{j \in I_k} \in S_k.
$$
\item $M_k = \{x^{(j)}= 0\,\,(j \in I_k)\}$ holds for $k \in \Theta$ (which follows from
	the above 1.~and 2.).
\end{enumerate}
\end{teo}

\

Note that, in the above coordinate systems, the morphisms $\iota_{k,j}$  $(j \preceq k)$
also becomes the canonical projection by {\eqref{eq:ap-xxx-1}}.
Furthermore, $\{I_k\}_{k \in \Theta}$ satisfies the conditions:
\begin{enumerate}
\item $i \preceq j$ $\iff$ $I_i \supset I_j$ ($i,j \in \Theta$).
\item For $i \ne j \in \Theta$, either $I_i \subsetneq I_j$, $I_j \subsetneq I_i$
	or $I_i \cap I_j = \emptyset$ holds.
\item $I_k \setminus \bigcup_{k \precneq j} I_j = \{k\}$ for any $k \in \Theta$.
\end{enumerate}

The proof of the theorem is the same as that of Proposition 1.2 \cite{HP13}, where
we use $\{\varpi_k\}$ as coordinate functions $\{f_{k,i}\}$.

Let $\varpi^1=\{\varpi^1_k:X \to S^1_k\}_{k \in \Theta}$ 
and $\varpi^2=\{\varpi^2_k:X \to S^2_k\}_{k \in \Theta}$ 
be compatible families of submersive morphisms with respect to $(\Theta,\preceq)$.
A morphism of compatible submersive systems from $\varpi^1$ to $\varpi^2$ 
is a set of morphisms
$\varphi_X: X \to X$ and $\varphi_{k}:S^1_k \to S^2_k$
$(k \in \Theta)$ satisfying 
$$
\varpi^2_k\circ \varphi_X = \varphi_k \circ \varpi^1_k\qquad (k \in \Theta).
$$
Note that, for $j \preceq k$,
$$
\iota^2_{k,j}\circ \varphi_j = \varphi_k \circ \iota^1_{k,j}
$$
follows from \eqref {eq:ap-xxx-1}.

We say that the above compatible submersive systems are equivalent if
there exists an isomorphism of compatible subversive systems
between $\varpi^1$ and $\varpi^2$.
Assume that $\varpi^1$ and $\varpi^2$ are equivalent.
Then it is easy to see that since $\varphi_X \times \varphi_X: X^2 \to X^2$ gives
an isomorphism of the submanifolds
$\{(z,w) \in X^2\,|\, \varpi^1_k(z) = \varpi^1_k(w)\}$
and
$\{(z,w) \in X^2\,|\, \varpi^2_k(z) = \varpi^2_k(w)\}$,
an equivalent class of compatible submersive systems 
determine the complex
$$
\muhom_{\widehat{\chi}_2}
(\mathbb{C}_{\Delta^{\varpi_1}},\mathbb{C}_{\Delta^{\varpi_2}},\dots,
	\mathbb{C}_{\Delta^{\varpi_\ell}};\,
\mathscr{O}^{(0,n)}_{X^2}),
$$
up to isomorphisms. 

\end{section}

\section*{Acknowledgement}
The first author is supported in part by JSPS KAKENHI Grant 21K03284. The second author is supported in part by Gruppo
Nazionale per l'Analisi Matematica, la Probabilità e loro Applicazioni GNAMPA - INdAM.

\addcontentsline{toc}{section}{

\noindent
\parbox[t]{.48\textwidth}{
Naofumi HONDA \\
Department of Mathematics, \\
Faculty of Science, \\
Hokkaido University, \\
060-0810 Sapporo, Japan. \\
honda@math.sci.hokudai.ac.jp } \hfill
\parbox[t]{.48\textwidth}{
Luca PRELLI\\
Dipartimento di Matematica \\
Universit\`{a} degli Studi di Padova,\\
Via Trieste 63,\\
35121 Padova, Italia. \\
lprelli@math.unipd.it }

\end{document}